\documentclass[a4paper,11pt,twoside]{amsart}
\usepackage[english]{babel}
\usepackage[utf8]{inputenc}

\usepackage[a4paper,inner=3cm,outer=3cm,top=4cm,bottom=4cm]{geometry}
\usepackage{fancyhdr}
\usepackage{bold-extra}
\usepackage{ mathrsfs }

\usepackage{comment}
\usepackage{graphics,color}
\usepackage{aliascnt}
\usepackage[hidelinks]{hyperref}

\usepackage{amsmath}
\usepackage{amsfonts}
\usepackage{amssymb}
\usepackage{amsthm}
\usepackage{comment}
\usepackage{mathtools}
\usepackage{enumerate}
\usepackage{enumitem}

\newcommand{\str}{\operatorname{str}}
\newcommand{\psr}{\operatorname{psr}}
\newcommand\sml{\operatorname{sml}}

\newcommand{\wh}{\widehat}
\newcommand{\wt}{\widetilde}
\newcommand{\Span}{\operatorname{Span}}
\newcommand{\CP}{\mathcal{P}}

\newcommand{\CR}{\mathcal{R}}
\newcommand{\CL}{\mathcal{L}}
\newcommand{\eps}{\varepsilon}
\newcommand{\ve}{\mathbf}

\def\F{\mathbb{F}}

\def\R{\mathbb{R}}
\def\C{\mathbb{C}}
\def\Z{\mathbb{Z}}
\def\E{\mathbb{E}}
\def\P{\mathcal{P}}
\def\L{\mathcal{L}}
\def\Q{\mathbb{Q}}

\def\T{\mathbb{T}}
\def\le{\leqslant}

	\newtheorem{thm}{\textbf{Theorem}}[section] 
	
	\newtheorem{prop}[thm]{\textbf{Proposition}}
	\newtheorem{lmm}[thm]{\textbf{Lemma}}

	\theoremstyle{definition}
	\newtheorem{defi}[thm]{\textbf{Definition}}  
	
\numberwithin{equation}{section}

\begin{document}
\author{Xuancheng Shao}
\address{Department of Mathematics, University of Kentucky, Lexington, KY 40506, USA}
\email{xuancheng.shao@uky.edu}

\author{Weiwen Zhang}
\address{School of Mathematical Sciences, University of Science and Technology of China, Hefei 230026, China}
\email{llcwwzhang@mail.ustc.edu.cn}

\title{Simultaneous popular differences for polynomial and linear progressions in finite fields}	
	
	\maketitle
	\begin{abstract}
		We study simultaneous popular differences for polynomial progressions and linear progressions over finite fields \(\mathbb F_p\). For certain families of polynomial progressions of algebraic complexity $1$, such as $(x, x+y, x+2y)$ and $(x, x+y^3, x+2y^3)$, we prove the existence of $y \neq 0$ which is popular simultaneously for both progressions. This extends previous results by Conlon, Dong and Hong who considered complexity $0$ progressions. On the other hand, we  construct counterexamples showing that simultaneous popular differences are not guaranteed to exist for $(x, x+y, x+2y)$ and $(x, x+y^2, x+2y^2)$, or for $(x, x+y, x+2y)$ and $(x, x+y, x+3y)$, or for $(x, x+y, x+2y)$ and $(x,x+y,x+2y,x+3y)$, answering questions posed by Conlon, Dong and Hong.
	\end{abstract}

	\section{Introduction}

Denote by $\Z_{0}[y]$ the subset of $\Z[y]$ consisting of polynomials with zero constant term. Let $\P=\{P_{1},\dots,P_{k}\}$ be a collection of nonzero polynomials in $\Z_{0}[y]$. We consider the polynomial progression associated with $\P$:
	\begin{equation}\label{eq:poly-prog}
	x,x+P_{1}(y),\cdots,x+P_{k}(y).
\end{equation}
The polynomial Szemer\'edi theorem of Bergelson and Leibman \cite{BL96}, which extends Szemer\'edi's theorem from arithmetic progressions to polynomial progressions, asserts that every subset of the integers of positive upper density contains a progression of the form \eqref{eq:poly-prog} with $y\ne0$. In this paper we work over $\F_p$, with the polynomials $\CP$ fixed and $p\to\infty$.

\subsection{Popular differences}

Let $\P=\{P_{1},\dots,P_{k}\}$ be a collection of nonzero polynomials in $\Z_{0}[y]$.  For a function $f: \F_p\rightarrow\C$ we define
$$
I_{\CP}(f;y) = \E_{x \in \F_p} f(x) \prod_{i=1}^k f(x + P_i(y))
$$
for $y \in \F_p$. Thus for $A \subset \F_p$ with density $\alpha$, the quantity $I_{\CP}(1_A;y)$ gives the normalized count for the number of polynomial progressions in $A$ with a given difference parameter $y$, and  the expected value of it in a random set $A$ of density $\alpha$ is $\alpha^{k+1}$ (assuming that $0, P_1(y),\cdots,P_k(y)$ are distinct).
If some $y \in \F_p\setminus\{0\}$ satisfies
$$
			I_{	\mathcal{P}}(1_A;y) \geq \alpha^{k+1}-\eps
$$
for some small $\eps > 0$, then we informally refer to $y$ as a \emph{popular difference} and say $\CP$ has the popular difference property.  Thus if $y$ is a popular difference then the progressions can occur with approximately their random-set frequencies at the difference parameter $y$.

The problem of determining the existence of popular differences started with linear progressions.
Green \cite{G05} proved that the $3$-AP pattern $x, x+y, x+2y$ has the popular difference property, using arithmetic regularity. Green and Tao \cite{GT10} proved the analogous result for $4$-APs. In fact, there is a complete classification of linear progressions satisfying the popular difference property.
Combined with the counterexamples of Sah, Sawhney and Zhao \cite{SSZ21}, these results show that a linear progression of the form
$$
x, x+c_1y, \cdots, x+c_ky,
$$
where $c_1<\cdots<c_k$ are distinct positive integers, has the popular difference property if and only if $k \leq 2$, or if $k=3$ and it is an additive quadruple (i.e. $c_3=c_1+c_2$). See also \cite{FP19} for quantitative popular differences in vector spaces and \cite{FPZ} for related questions concerning common and Sidorenko linear equations.

For polynomial progressions, it is known that a polynomial progression with algebraic complexity at most $1$ satisfies the popular difference property. 
The notion of complexity, in the setting of polynomial progressions over finite fields, was introduced and studied by Kuca \cite{K21,K23}.

\begin{defi}[Algebraic relations and algebraic complexity]
	  Let $\P=\{P_{1},\dots,P_{k}\}$ be a collection of nonzero polynomials in $\Z_{0}[y]$. An \emph{algebraic relation} of degree $d$ satisfied by the polynomial progression \eqref{eq:poly-prog}  is a polynomial tuple $(Q_{0},\dots,Q_{k})\in\R[u]^{k+1}$ with $\max(\deg Q_0,\cdots,\deg Q_k) = d$ satisfying
		\begin{equation*}
			Q_{0}(x)+Q_{1}(x+P_{1}(y))+\dots+Q_{k}(x+P_{k}(y))=0.
		\end{equation*} 
The \emph{algebraic complexity} of $\CP$ is defined to be the smallest non-negative integer $s$ such that any algebraic relation satisfied by the polynomial progression has degree at most $s$.
	\end{defi}

Later we will also use the notion of algebraic complexity of a collection of linear forms in several variables $\ve{y} = (y_1,\cdots,y_m)$. The definition in this setting is the same with the scalar $y$ replaced by the vector $\ve{y}$.

In particular, $\CP$ has algebraic complexity $0$ precisely when the polynomials in $\P$ are linearly independent, and $\CP$ has algebraic complexity at most $1$ when there are no quadratic algebraic relations. Examples of complexity-one progressions include
\[
(x,x+y,x+2y),\ \
(x,x+y,x+y^2,x+y+y^2),
\ \
(x,x+y,x+2y,x+y^3,x+2y^3).
\]
On the other hand, the progression $x, x+y, x+2y, x+y^2$ has complexity $2$.

For $\CP = \{P_1,\cdots,P_k\}$ with complexity $0$, Peluse \cite{P18,P19} proved that the total number of polynomial progressions in any subset $A \subset \F_p$ with density $\alpha$ is $(\alpha^{k+1} + O(p^{-c}))p^2$ for some constant $c = c(\CP) > 0$. This immediately implies the existence of popular differences by the pigeonhole principle. Kuca \cite{K21,K23} extended the analysis to polynomial progressions $\CP$ of complexity $1$ and established the popular difference property for all such $\CP$. See \cite[Theorem~1.16(iii)]{K23}.

\begin{thm}\label{ac1}
	 Let $\P$ be a finite collection of nonzero polynomials in $\Z_{0}[y]$ with algebraic complexity at most $1$. Then for any $\eps>0$ and $A\subset\F_{p}$ with density $\alpha>0$, there are at least $cp$ values of $y \in \F_p$ for some constant $c = c(\P, \eps) > 0$, such that
	 $$
	 I_{\CP}(1_A;y) \geq \alpha^{|\CP|+1}-\eps.
	 $$
	\end{thm}
	
In the integer setting, popular differences results for polynomial progressions were obtained in \cite{LM13} for two-point patterns, in \cite{SW25} for distinct degree polynomials, and in \cite{PPS24} for a two-dimensional nonlinear Roth pattern.

\subsection{The existence of simultaneous popular differences}

Conlon, Dong and Hong \cite{CDH} introduced the following simultaneous popular difference problem. Given two or more polynomial progressions, is it possible to find a nonzero $y$ which is a popular difference simultaneously for all these progressions?
For a collection $\P$ of linearly independent polynomials in $\Z_0[y]$, they proved the existence of a simultaneous popular difference for all subprogressions of $\CP$.

In this paper, we initiate the study of simultaneous popular differences for certain polynomial progressions of complexity $1$. All spans below are over $\Q$.

\begin{thm}\label{Main}
		Let $\P$ be a finite collection of nonzero polynomials in $\Z_{0}[y]$ with algebraic complexity at most $1$. Let $\P_{1},\dots,\P_{\ell}\subset\P$ be subsets satisfying
$$
\Span(\CP_1\cup\cdots\cup \CP_\ell) = \Span(\CP_1)\oplus \cdots\oplus \Span(\CP_\ell).
$$		
		Then for any $\eps>0$ and $A\subset\F_{p}$ with density $\alpha>0$, there are at least $cp$ values of $y \in \F_p$ for some constant $c = c(\CP,\eps) > 0$, such that
$$
I_{\CR}(1_A;y) \geq \alpha^{|\CR|+1} - \eps
$$
simultaneously for all $\CR \in \{\CP_1,\cdots,\CP_{\ell}\}$ and all $\CR\subset \CP$ with algebraic complexity $0$.
	\end{thm}

For example, taking $\P=\{y,2y,y^3,2y^3\}$, $\CP_1 = \{y,2y\}$, and $\CP_2 = \{y^3,2y^3\}$ gives simultaneous popular differences for
\[
(x,x+y,x+2y),\ \ (x,x+y^3,x+2y^3), \ \ (x, x+y, x+y^3).
\]
The direct sum condition in the hypothesis of Theorem \ref{Main} is not a necessary condition. In Theorems \ref{Main2} and \ref{Main3} we prove the simultaneous popular difference property for certain families of polynomial progressions not satisfying the direct sum hypothesis.

The next result concerns copies of the same linear pattern evaluated at different polynomials, and all the unions of the copies. 

\begin{thm}\label{Main2}
Let $\CP_1,\cdots,\CP_m$ be finite collections of nonzero polynomials in $\Z_0[y]$ of the form
$$
\CP_i = \{L_1(Q_i(y)), \cdots, L_k(Q_i(y))\}
$$
for each $1 \leq i \leq m$, where $Q_1,\cdots,Q_m \in \Z_0[y]$ are linearly independent polynomials and $L_1,\cdots,L_k \in \Z_0[y]$ are linear forms. Suppose that $\CP_1\cup\cdots\cup\CP_m$ has
 algebraic complexity at most $1$. Then for any $\eps>0$ and $A\subset\F_{p}$ with density $\alpha>0$, there are at least $cp$ values of $y \in \F_p$ for some constant $c = c(\CP_1,\cdots,\CP_m,\eps) > 0$, such that
$$
I_{\CP_I}(1_A;y) \geq \alpha^{|\CP_I|+1} - \eps
$$
simultaneously for all nonempty $I \subset \{1,2,\cdots,m\}$, where $\CP_I = \cup_{i \in I}\CP_i$.
	\end{thm}

For example, taking $\CP_1=\{y,2y\}$ and $\CP_2=\{y^3,2y^3\}$ gives simultaneous popular differences for 
\[
(x,x+y,x+2y),\ \ (x,x+y^3,x+2y^3),
\ \ (x,x+y,x+2y,x+y^3,x+2y^3).
\]
We also consider the following two pairs not covered by either Theorem \ref{Main} or \ref{Main2}.

\begin{thm}\label{Main3}
	Let $\CP_1,\CP_2$ be either one of the following pairs of collections:
	\begin{enumerate}
\item $\CP_1 = \{y, 2y\}$ and $\CP_2 = \{y, 2y, y^3\}$;
\item $\CP_1 = \{y^2, 2y^2\}$ and $\CP_2 = \{y, y^2, y+y^2\}$.
	\end{enumerate}	
	Then for any $\eps > 0$ and $A \subset \F_p$ with density $\alpha > 0$, there are at least $cp$ values of $y \in \F_p$ for some constant $c = c(\eps) > 0$, such that
	$$
	I_{\CP_i}(1_A;y) \geq \alpha^{|\CP_i|+1}-\eps
	$$
simultaneously for each $i \in \{1,2\}$.
	\end{thm}

\subsection{The nonexistence of simultaneous popular differences}

Conlon, Dong and Hong \cite{CDH} proved that, in the setting of $\F_p^n$ with $p$ sufficiently large but fixed and $n\to\infty$, the two linear patterns 
$$
(x, x+y, x+2y), \ \ (x, x+2y, x+4y)
$$
do not have the simultaneous popular difference property. In particular, for fixed $p$ they constructed subsets $A \subset \F_p^n$ of density $1/2 + o_{n\to\infty}(1)$, such that either $y$ or $2y$ is not a popular difference for $3$-APs in $A$ for any $y \in \F_p^n\setminus\{0\}$. They asked about the pair
\[
(x,x+y,x+2y),\ \ (x,x+y^2,x+2y^2)
\]
and also other pairs of linear patterns. We answer these questions in the setting of $\F_p$ as $p\rightarrow\infty$. We first consider the simpler pair consisting of a $3$-AP and a two-point quadratic polynomial progression.

\begin{thm}\label{counterex}
	Let  $\CP_1 = \{y, 2y\}$ and $\CP_2 = \{y^2\}$.  There are absolute constants $\alpha,c>0$, such that for sufficiently large prime $p$ there exists  a subset $A \subset \F_p$ with density at least $\alpha$ satisfying
$$
\min\Big\{I_{\CP_1}(1_A;y) - \alpha^{3}, I_{\CP_2}(1_A;y) - \alpha^{2}\Big\} \leq -c
$$
for all $y \in \F_p\setminus\{0\}$.
\end{thm}

Note that in this case $\CP_1\cup\CP_2$ has algebraic complexity $2$. This shows that the hypothesis of $\CP$ having complexity $1$ in Theorem \ref{Main} cannot be weakened (without imposing additional assumptions).

Next we consider two $3$-APs with common differences $y$ and $y^2$.

\begin{thm}\label{counterex2}
 Let $\CP_1 = \{y, 2y\}$ and $\CP_2 = \{y^2, 2y^2\}$.
 For sufficiently small $\alpha>0$ and prime $p$ sufficiently large in terms of $\alpha$, one can find a subset $A \subset \F_p$ with $|A| \geq \alpha p$ such that
 $$
 \min\Big\{I_{\CP_1}(1_A;y), I_{\CP_2}(1_A;y)\Big\} \leq \exp(-c(\log \alpha^{-1})^2)
 $$
for some absolute constant $c>0$ and all $y \in \F_p\setminus\{0\}$.
\end{thm}

Note that the upper bound here is smaller than $\alpha^A$ for any given $A > 0$, provided that $\alpha$ is small enough. Thus in this case either $y$ or $y^2$ is the common difference of very few $3$-APs in $A$ for every nonzero $y$.

Next we consider pairs of linear patterns. For linear patterns, we will represent them by their coefficient sets. For a finite coefficient set $C \subset \Z$, the associated linear pattern is $(x+cy)_{c \in C}$. For a function $f: \F_p\rightarrow\C$ we define
$$
I_C(f;y ) = \E_{x \in \F_p} \prod_{c \in C} f(x+cy).
$$

Two sets $C,D\subset \Z$ are said to be \emph{congruent} if $C$ is a translate of $D$ or $-D$; i.e. if $C = t+D$ or $C = t-D$ for some $t \in \Z$. We say that $C\subset \Z$ is a \emph{triple} if $|C|=3$. We say that $C \subset \Z$ is an \emph{additive quadruple} if $C$ is congruent to $\{0,a,b,a+b\}$ for some positive integers $a<b$. 

As mentioned above, it is known that $C$ satisfies the popular difference property if and only if $|C| \leq 3$ or $C$ is an additive quadruple \cite{SSZ21}. Thus if a pair $C_1,C_2$ of linear patterns satisfies the simultaneous popular difference property, then for each $i$ we have either $|C_i| \leq 3$ or $C_i$ is an additive quadruple. If one of $C_1,C_2$ has at most two points and the other has at most three points, then one of $C_1,C_2$ has complexity $0$ while the other pattern has complexity at most $1$, and hence they have the simultaneous popular difference property by Proposition \ref{Main-prop} below. We now consider $C_1,C_2$ when each $C_i$ is either a triple or an additive quadruple. The next statement gives counterexamples when they are not congruent.

\begin{thm}\label{no-spd-linear}
Let $C_1,C_2 \subset \Z$ with each $C_i$ either a triple or an additive quadruple. Suppose that $C_1,C_2$ are not congruent. There are constants $\alpha,c>0$ depending on $C_1,C_2$, such that for sufficiently large prime $p$ there exists  a subset $A \subset \F_p$ with density at least $\alpha$ satisfying
$$
\min\Big\{I_{C_1}(1_A;y) - \alpha^{|C_1|}, I_{C_2}(1_A;y) - \alpha^{|C_2|}\Big\} \leq -c
$$
for all $y \in \F_p\setminus\{0\}$.
\end{thm}

For example, the pair $(x, x+y, x+2y)$ and $(x, x+y, x+3y)$ does not have the simultaneous popular difference property. The pair of a $3$-AP and a $4$-AP also does not have the simultaneous popular difference property.

When $C_1,C_2$ are congruent additive quadruples, they must be translates of each other, and hence they satisfy the simultaneous popular difference property for obvious reasons. The same holds for congruent triples $C_1,C_2$ which are translates of each other. For congruent triples that are not translates of each other, for example $(x, x+y, x+3y)$ and $(x, x+2y, x+3y)$, we do not yet have an answer. The only remaining case not covered by Theorem \ref{no-spd-linear} and the discussions above is when one of $C_1,C_2$ has at most two points and the other is an additive quadruple. In this case, we expect that the simultaneous popular difference property follows by using the $U^3$ regularity argument from \cite{GT10}.

\subsection{Proof outline and organization of the paper}

The argument for establishing the existence of simultaneous popular differences in the proof of Theorems \ref{Main}, \ref{Main2}, and \ref{Main3} has two parts: a comparison between polynomial and linear progressions in several variables (Section \ref{sec:poly-to-linear}), followed by a simultaneous popular difference argument for the linear systems (Section \ref{sec:spd}).

For the comparison step in Section \ref{sec:poly-to-linear}, taking the three polynomial progressions
$$
(x, x+y, x+2y), \ \ (x, x+y^3, x+2y^3), \ \ (x, x+y, x+y^3)
$$
as an example, we prove in Theorem \ref{poly-to-linear} that simultaneous popular differences for the associated linear system in two variables $y,z$
$$
(x, x+y, x+2y), \ \ (x, x+z, x+2z), \ \ (x, x+y, x+z)
$$
transfer to the original polynomial progression. This step makes use of the assumption that the union $\{y, 2y, y^3, 2y^3\}$ has algebraic complexity $1$, and proceeds by applying Kuca's counting lemma \cite{K23} to certain mixed moments of the counting functions and comparing their joint distributions.

For the linear problem in Section \ref{sec:spd}, we use a $U^2$ arithmetic regularity decomposition to produce a weight $\mu(y,z)$ which is a smoothed version of the product $B \times B$ for some Bohr neighborhood $B$. We prove in Proposition \ref{prop:linear-0} that almost all of $(y,z) \in B \times B$ are popular for the complexity $0$ system $(x, x+y, x+z)$, and in Proposition \ref{prop:linear-1} that a positive proportion of $(y,z) \in B \times B$ are popular for each of the two individual complexity $1$ systems $(x, x+y, x+2y)$ and $(x, x+z, x+2z)$. Since these two systems involve separate variables (which holds under the direct sum hypothesis in Theorem \ref{Main}), we have a positive proportion of $(y,z) \in B \times B$ that are simultaneous popular differences. For Theorems \ref{Main2} and \ref{Main3}, we use an additional Ramsey argument and averaging inequalities to produce the simultaneous popular differences.

For the counterexamples in Theorems \ref{counterex}, \ref{counterex2}, and \ref{no-spd-linear}, we use two types of construction. Theorems \ref{counterex} and \ref{no-spd-linear} use a Fourier construction to define $f: \F_p\rightarrow [0,1]$ by
\[
f(x)=\frac12+G\Big(\frac{x}{p},\frac{x^2}{p}\Big),
\]
where $G$ is a real trigonometric polynomial, viewed as a small perturbation. The counting operator can  be written in terms of Fourier coefficients of $G$, and the task is then reduced to prescribing the Fourier coefficients appropriately. 

Theorem \ref{counterex2} uses a thickened Behrend construction $B$ (which is Freiman isomorphic to $\{1,2,\cdots,K\} \times \wt{B}$ for some large $K$ and some large $\wt{B}$ without nontrivial $3$-APs \cite{B46}) and 
\[
A=\{x\in\F_p:x\in B,\ x^2\in B\}.
\]
If $A$ contains $3$-APs with common difference $y^2$, then the $3$-AP-free nature of $\wt{B}$ forces $y^2$ into a short interval of length $K$. If $A$ also contains $3$-APs $x, x+y, x+2y$ with common difference $y$, then the identity
\[
x^2+(x+2y)^2-2(x+y)^2=2y^2
\]
makes $x^2, (x+y)^2, (x+2y)^2$ an approximate $3$-AP in $B$, which forces them to lie in the same short interval of length $K$. This restricts to very few possible values of $x$.

The details of these constructions are given in Sections \ref{sec:nspd1}, \ref{sec:nspd2}, and \ref{sec:nspd3} for Theorems \ref{counterex}, \ref{counterex2}, and \ref{no-spd-linear}, respectively.

\subsection*{Acknowledgements}
The majority of this work was carried out when WZ was visiting the Department of Mathematics at the University of Kentucky, whose hospitality is greatly acknowledged. XS was supported by NSF grant DMS-2452462. WZ was supported by USTC Fellowship Grant No. U19582026.

\subsection*{AI assistance}	
ChatGPT (OpenAI) was used by the authors to assist with polishing
and proofreading. The authors take full responsibility for the mathematical statements,
arguments and their presentation.
	
	\section{Preliminaries}

In this section we record the key ingredients used in finding simultaneous popular differences in Theorems \ref{Main}, \ref{Main2}, and \ref{Main3}: an arithmetic regularity lemma for the Gowers $U^2$ norm, a counting lemma for polynomial progressions with algebraic complexity at most $1$, and a twisted generalized von Neumann theorem for linear patterns in several variables with algebraic complexity at most $1$.

Green \cite{G05} introduced an arithmetic analogue of Szemer\'edi's regularity lemma for abelian groups and used it to prove the popular difference property for $3$-APs. This arithmetic regularity lemma was later generalized by Green and Tao \cite{GT10} for higher order Gowers norms. In the $U^2$ version of the regularity lemma that we use, an arbitrary $1$-bounded function is decomposed into a structured part and some errors, where the structured part is approximately invariant under translations by elements from a Bohr set.

For functions $f,g :\F_p\rightarrow \C$, we define the Fourier transform $\wh{f}$  and the convolution $f*g$ by
$$
\wh f(\xi)=\E_{x\in\F_p}f(x)e_p(-\xi x),\qquad
f*g(x)=\E_{u\in\F_p}f(u)g(x-u),
$$
where $e_p(x)=e(x/p)$ and $e(t) = \exp(2\pi it)$. The Gowers $U^2$-norm of $f$ is closely related to the Fourier transform:
\[
\|f\|_{U^2}^4=\sum_{\xi\in\F_p}|\wh f(\xi)|^4,
\qquad
\|\wh f\|_{\infty}\leq\|f\|_{U^2}
\leq\|\wh f\|_{\infty}^{1/2}
\]
for every $1$-bounded function $f$.
For a finite set $\Gamma\subset\F_p$ and $\rho \in (0,1/2)$, define
\[
B(\Gamma,\rho)=\Big\{x\in\F_p:
\Big\|\frac{\xi x}{p}\Big\|_{\T}\leq\rho
\text{ for every }\xi\in\Gamma\Big\},
\]
where $\|t\|_{\T}$ denotes the distance from $t$ to the nearest integer. The standard Bohr set bound gives
$|B(\Gamma,\rho)|\geq\rho^{|\Gamma|}p$; see \cite[Lemma~4.20]{TV}.

The following is the specialization of \cite[Proposition~3.3]{CDH} that we need. A growth function $\eta$ is an increasing function $\eta:\R_+\to\R_+$.

\begin{prop}\label{regularity}
Let $0<\eps<1$, let $\eta$ be a growth function, and let $f:\F_p\to[0,1]$. There exist a finite set $\Gamma\subset\F_p$ with $|\Gamma| \ll_{\eps,\eta}1$, a radius $\rho\gg_{\eps,\eta}1$, and a decomposition
\[
f=f_{\str}+f_{\psr}+f_{\sml}
\]
such that:
\begin{enumerate}[label=\textup{(\alph*)},itemsep=0pt]
\item $f_{\str},f_{\psr},f_{\sml}$ are all $1$-bounded;
\item $f_{\str}\geq0$, $\E f_{\str}=\E f$, and
\[
f_{\str}(x+d)=f_{\str}(x)+O(\eps)
\quad\text{for all }x\in\F_p,\ d\in B(\Gamma,\rho);
\]
\item $\|\wh{f_{\psr}}\|_{\infty}\leq1/\eta(|\Gamma|+\rho^{-1})$;
\item $\|f_{\sml}\|_2\leq\eps$.
\end{enumerate}
\end{prop}

The growth function allows the pseudorandom error to be made small relative to the Bohr set obtained from the decomposition. Observe that if $L(\ve{y}) \in \Z_0[\ve{y}]$ is a linear form and all coordinates of $\ve{y}$ lie in $B$, then
\begin{equation}\label{eq:translation-invariance}
f_{\str}(x+L(\ve{y})) = f_{\str}(x) + O_L(\eps)
\end{equation}
for all $x \in \F_p$.

Next we turn to counting theorems for polynomial progressions in finite fields. In the complexity $0$ case when the polynomials $\CP = \{P_1,\cdots,P_k\}$ are linearly independent, Peluse \cite{P18,P19} proved asymptotic counting formulas for $1$-bounded functions $f_0,\cdots,f_k:\F_p\rightarrow\C$ with a power-saving error:
$$
\E_{x,y \in \F_p} f_0(x) \prod_{i=1}^k f_i(x+P_i(y)) = \prod_{i=0}^k (\E f_i) + O(p^{-c}),
$$
where $c>0$ is a constant depending on $\CP$.
This is in general not the case when the polynomials satisfy linear relations.
For progressions of algebraic complexity $1$, Kuca \cite[Corollary 1.15(i)]{K23} proved a counting theorem which identifies the main term with the count of an associated linear progression in several variables.

\begin{thm}\label{counting lem}
Let $\P=\{P_1,\dots,P_k\}$ be a collection of nonzero polynomials in $\Z_0[y]$ with algebraic complexity at most $1$. Let $Q_1,\dots,Q_m\in\Z_0[y]$ be linearly independent polynomials such that we can write
\[
P_j=\sum_{r=1}^m a_{jr}Q_r,\qquad a_{jr}\in\Q
\]
for each $1 \leq j \leq k$.
Define linear forms $L_1,\cdots,L_k$ in $m$ variables by
$$
L_j(\ve{y})=\sum_{r=1}^m a_{jr}y_r
$$
for each $1\leq j\leq k$ and $\ve{y} = (y_1,\cdots,y_m)$. Then for all $1$-bounded functions $f_0,\dots,f_k:\F_p\to\C$ we have
\[
\E_{x,y\in\F_p}f_0(x)\prod_{i=1}^k f_i(x+P_i(y))
=
\E_{x\in \F_p, \ve{y} \in \F_p^m}f_0(x)
\prod_{i=1}^k f_i(x+L_i(\ve{y}))+o(1),
\]
where $o(1)$ denotes a quantity that tends to $0$ as $p\rightarrow\infty$, which depends on $\CP$ but not on $f_0,\cdots,f_k$.
\end{thm}

The minor extension from $a_{jr} \in\Z$ in Kuca's statement to $a_{jr} \in \Q$ is easily achieved by clearing denominators and rescaling the $x$-variable. The linear independence of $Q_1,\cdots,Q_m$ ensures that the associated linear forms $L_1,\cdots,L_k$ retain exactly the linear relations satisfied by $P_1,\dots,P_k$. This linear independence assumption seems to be missed in Kuca's statement.

In our treatment of linear patterns in several variables, we need a twisted version of the generalized von Neumann theorem for the associated counting operator.

\begin{prop}\label{prop:twisted-gvn}
Let $\CL = (L_1,\cdots,L_k)$ be a finite collection of nonzero linear forms in $m$ variables in $\Z_0[\ve{y}]$ with algebraic complexity at most $1$. Let $f_0,f_1,\cdots,f_k:\F_p\rightarrow\C$ be $1$-bounded. Then
$$
\Big|\E_{x,\ve{y}} e_p(\xi\cdot \ve{y})f_0(x) \prod_{i=1}^kf_i(x+L_i(\ve{y}))\Big| \ll_{\CL} \min_{0 \leq i \leq k}\|f_i\|_{U^2}^c
$$
for any $\ve{\xi} \in \F_p^m$ and some constant $c = c(\CL) > 0$.
\end{prop}

The twist $e_p(\xi\cdot \ve{y})$ can be easily removed by using the structure of the linear forms. The untwisted version belongs to the theory of true complexities of linear forms initiated by Gowers and Wolf \cite{GW10,GW11}. Green and Tao \cite{GT10} proved a qualitative result under the flag condition, and Altman \cite{A22} removed this restriction. We use Manners' quantitative result \cite[Theorem 1.1.5]{M21}. In the single variable case $m=1$, this can be proved directly via Fourier analysis.  

\begin{proof}
If the linear form $\xi\cdot \ve{y}$ does not lie in the span of $L_1,\cdots,L_k$ over $\F_p$, then there exists $\ve{v} \in \F_p^m$ such that $L_1(\ve{v}) = \cdots = L_k(\ve{v}) = 0$ but $\xi\cdot \ve{v} \neq 0$. Translating the variable $\ve{y}$ in the average by $\ve{v}$, we obtain
$$
\E_{x,\ve{y}} e_p(\xi\cdot \ve{y})f_0(x) \prod_{i=1}^kf_i(x+L_i(\ve{y})) = e_p(\ve{\xi}\cdot\ve{v}) \E_{x, \ve{y}} e_p(\xi\cdot \ve{y})f_0(x) \prod_{i=1}^kf_i(x+L_i(\ve{y})),
$$
which implies that both sides above are $0$, and the conclusion holds trivially.

Now assume that $\xi\cdot \ve{y} = \sum_i b_iL_i(\ve{y})$ for some $b_1,\cdots,b_k \in \F_p$. We can write the left-hand side as
$$
\E_{x,\ve{y}} e_p(\xi\cdot \ve{y})f_0(x) \prod_{i=1}^kf_i(x+L_i(\ve{y})) = \E_{x,\ve{y}} g_0(x) \prod_{i=1}^k g_i(x + L_i(\ve{y})),
$$
where
$$
g_0(x) = f_0(x) e_p(-(b_1+\cdots+b_k)x), \ \ g_i(x) = f_i(x) e_p(b_ix)
$$
for $1 \leq i \leq k$. Since $\|g_i\|_{U^2} = \|f_i\|_{U^2}$ for each $0 \leq i \leq k$, it suffices to prove Proposition \ref{prop:twisted-gvn} with $\xi=0$, in which case the conclusion follows from \cite[Theorem 1.1.5]{M21}.
\end{proof}

\section{From polynomial progressions to linear progressions}\label{sec:poly-to-linear}

The purpose of this section is to reduce the task from dealing with polynomial progressions to dealing with linear progressions in several variables. This reduction is summarized in the following theorem and requires the assumption that the ambient set of polynomials $\CP$ has algebraic complexity at most $1$.

	\begin{thm}\label{poly-to-linear}
Let $\CP = \{P_1,\cdots,P_k\}$ be a collection of nonzero polynomials in $\Z_0[y]$ with algebraic complexity at most $1$. Let $Q_1,\cdots,Q_m \in \Z_0[y]$ be linearly independent polynomials such that we can write
$$
P_j = \sum_{r=1}^m a_{jr} Q_r, \qquad a_{jr} \in \Q
$$
for each $1 \leq j \leq k$. Define linear forms $L_1,\cdots,L_k$ in $m$ variables by
$$
L_j(\ve{y}) = \sum_{r=1}^m a_{jr}y_r
$$
for each $1 \leq j \leq k$ and $\ve{y}=(y_1,\cdots,y_m)$. Let $\CR_1,\cdots,\CR_q \subset \CP$ and  define $\CL_a = \{L_j: P_j \in \CR_a\}$ for $1 \leq a \leq q$.
For any $\eps > 0$ and $A \subset \F_p$ with density $\alpha > 0$, if there are at least $\beta p^m$ values of $\ve{y} \in \F_p^m$ for some $\beta > 0$ such that
$$
I_{\CL_a}(1_A;\ve{y}) \geq \alpha^{|\CL_a|+1} - \eps
$$
simultaneously for all $1 \leq a \leq q$, then there are at least $(\beta - o(1))p$ values of $y \in \F_p$ such that
$$
I_{\CR_a}(1_A;y) \geq \alpha^{|\CR_a|+1} - 2\eps
$$
simultaneously for all $1 \leq a \leq q$.
	\end{thm}
	
For a finite collection $\CL \subset \Z_0[\ve{y}]$ of linear forms in the $m$ variables $\ve{y} = (y_1,\cdots,y_m)$ with zero constant terms, the counting operator is defined by
$$
I_{\CL}(f;\ve{y}) = \E_{x \in \F_p} f(x) \prod_{L \in \CL}f(x + L(\ve{y}))
$$
for $f: \F_p\rightarrow \C$. If $I_{\CL}(1_A;\ve{y}) \geq \alpha^{|\CL|+1}-\eps$, we informally refer to $\ve{y}$ as a popular difference for $\CL$.

Note that if $\CP = \{P_1,\cdots,P_k\}$ has algebraic complexity at most $1$, then the corresponding linear pattern $\CL = \{L_1,\cdots,L_k\}$ as defined in the statement of Theorem \ref{poly-to-linear} also has algebraic complexity at most $1$. Thus, in view of Theorem \ref{poly-to-linear}, our main results (Theorems \ref{Main},  \ref{Main2}, and \ref{Main3}) can be respectively reduced to Propositions \ref{Main-prop}, \ref{Main2-prop}, and \ref{Main3-prop} below.
	
	\begin{prop}\label{Main-prop}
		Let $\L$ be a finite collection of nonzero linear forms in $m$ variables in $\Z_0[\ve{y}]$. Let $\L_{1},\dots,\L_{\ell}\subset\L$ be subsets, each having algebraic complexity $1$, satisfying
$$
\Span(\L_1\cup\cdots\cup \L_\ell) = \Span(\L_1)\oplus \cdots\oplus \Span(\L_\ell).
$$		
		Then for any $\eps>0$ and $A\subset\F_{p}$ with density $\alpha>0$, there are at least $cp^m$ values of $\ve{y} \in \F_p^m$ for some constant $c = c(\CL, \eps) > 0$, such that
		$$
		I_{\CR}(1_A;\ve{y}) \geq \alpha^{|\CR|+1}-\eps
		$$
simultaneously for all $\CR \in \{\L_1,\cdots,\L_\ell\}$ and all $\CR \subset \L$ with algebraic complexity $0$.
	\end{prop}
	
	\begin{prop}\label{Main2-prop}
Let $\CL = \{L_1,\cdots,L_k\}$ be a finite collection of nonzero linear forms in a single variable in $\Z_0[y]$ with algebraic complexity at most $1$. Let $m$ be a positive integer. For a nonempty subset $I \subset [m]$, let $\CL_I$ be the collection of linear forms in $m$ variables $\ve{y} = (y_1,\cdots,y_m)$ consisting of $L_j(y_i)$ with $1 \leq j \leq k$ and $i \in I$. Then for any $\eps>0$ and $A\subset\F_{p}$ with density $\alpha>0$, there are at least $cp^m$ values of $\ve{y} \in \F_p^m$ for some constant $c = c(\CL, \eps) > 0$, such that
		$$
		I_{\CL_I}(1_A;\ve{y}) \geq \alpha^{|\CL_I|+1}-\eps
		$$
simultaneously for all nonempty subsets $I \subset [m]$.
	\end{prop}
	
	\begin{prop}\label{Main3-prop}
	Let $\CL_1$ and $\CL_2$ be either one of the following pairs of linear forms in two variables $y,z$:
	\begin{enumerate}
\item $\CL_1 = \{y, 2y\}$ and $\CL_2 = \{y, 2y, z\}$;
\item $\CL_1 = \{z, 2z\}$ and $\CL_2 = \{y, z, y+z\}$.
	\end{enumerate}	
	Then for any $\eps > 0$ and $A \subset \F_p$ with density $\alpha>0$, there are at least $cp^2$ values of $(y,z) \in \F_p^2$ for some constant $c = c(\eps) > 0$, such that
		$$
		I_{\CL_i}(1_A;y,z) \geq \alpha^{|\CL_i|+1}-\eps
		$$
simultaneously for each $i \in \{1,2\}$.
	\end{prop}
		
In the rest of this section, we prove Theorem \ref{poly-to-linear}. The proofs of Propositions \ref{Main-prop}, \ref{Main2-prop}, and \ref{Main3-prop} are given in Section \ref{sec:spd}.		
		
\subsection{Joint distribution of $I_{\CR_a}$ and $I_{\CL_a}$}

Lying at the heart of Theorem \ref{poly-to-linear} is the following result connecting the joint distribution of $I_{\CR_a}$ for $1 \leq a \leq q$ with that of $I_{\CL_a}$.

\begin{prop}\label{joint}
Let the notations and assumptions be as in the statement of Theorem \ref{poly-to-linear}. Let $f: \F_p \rightarrow \C$ be a $1$-bounded function. For any fixed non-negative integers $n_1,\cdots,n_q$ we have
$$
\E_y \prod_{a=1}^q I_{\CR_a}(f;y)^{n_a} = \E_{\ve{y}}  \prod_{a=1}^q I_{\CL_a}(f;\ve{y})^{n_a} + o(1).
$$
Moreover, if $f$ takes values in $[0,1]$, then for any fixed continuous function $\Phi: [0,1]^q\rightarrow\R$ we have
$$
\E_y \Phi(I_{\CR_1}(f;y),\cdots, I_{\CR_q}(f;y)) = \E_{\ve{y}}  \Phi(I_{\CL_1}(f;\ve{y}), \cdots, I_{\CL_q}(f;\ve{y})) + o(1).
$$
\end{prop}

\begin{proof}
Expanding the product using the definition of $I_{\CR_a}(f;y)$, we can write
$$
\E_y \prod_{a=1}^q I_{\CR_a}(f;y)^{n_a} = \E_{\ve{x}} \E_y \prod_{a=1}^q \prod_{u=1}^{n_a} f(x_{a,u}) \prod_{P \in \CR_a} f(x_{a,u} + P(y)).
$$
Introducing an additional average over $h \in \F_p$ and shifting $\ve{x}$ by $h$, we obtain
$$
\E_y \prod_{a=1}^q I_{\CR_a}(f;y)^{n_a} = \E_{\ve{x}} \E_{h,y} \prod_{a=1}^q \prod_{u=1}^{n_a} f(x_{a,u}+h) \prod_{P \in \CR_a} f(x_{a,u}+h + P(y)).
$$
Since the polynomial $P$ above must be from $\CP=\{P_1,\cdots,P_k\}$, for fixed $\ve{x},h,y$ we may reorganize the product above as
$$
\Big(\prod_{a=1}^q\prod_{u=1}^{n_a} f(x_{a,u}+h)\Big) \Big(\prod_{j=1}^k \prod_{a: P_j \in \CR_a} \prod_{u=1}^{n_a} f(x_{a,u} + h + P_j(y))\Big) = F_{\ve{x},0}(h) \prod_{j=1}^k F_{\ve{x},j}(h+P_j(y)),
$$
where
$$
F_{\ve{x},0}(h) = \prod_{a=1}^q\prod_{u=1}^{n_a} f(x_{a,u}+h), \qquad F_{\ve{x},j}(h) = \prod_{a: P_j \in \CR_a} \prod_{u=1}^{n_a} f(x_{a,u} + h)
$$
for $1 \leq j \leq k$. Thus we have
$$
\E_y \prod_{a=1}^q I_{\CR_a}(f;y)^{n_a} = \E_{\ve{x}} \E_{h,y}F_{\ve{x},0}(h) \prod_{j=1}^k F_{\ve{x},j}(h+P_j(y)).
$$
The same argument applied to moments of $I_{\CL_a}(f;\ve{y})$ gives
$$
\E_{\ve{y}} \prod_{a=1}^q I_{\CL_a}(f;\ve{y})^{n_a} = \E_{\ve{x}} \E_{h,\ve{y}}F_{\ve{x},0}(h) \prod_{j=1}^k F_{\ve{x},j}(h+L_j(\ve{y})).
$$
Since $\CP$ has algebraic complexity at most $1$, we may apply Theorem \ref{counting lem} to conclude for each $\ve{x}$ that 
$$
\E_{h,y} F_{\ve{x},0}(h) \prod_{j=1}^k F_{\ve{x},j}(h+P_j(y)) = \E_{h,\ve{y}} F_{\ve{x},0}(h) \prod_{j=1}^k F_{\ve{x},j}(h+L_j(\ve{y})) + o(1).
$$
This establishes the first claim. The second claim about continuous $\Phi$ follows by approximating $\Phi$ by polynomials via the Weierstrass theorem.
\end{proof}

\subsection{Deduction of Theorem \ref{poly-to-linear} from Proposition \ref{joint}}

Let $\varphi: \R\rightarrow [0,1]$ be a continuous approximation to the indicator function of $[-2\eps,+\infty)$, with the property that
$$
1_{\{t \geq -\eps\}} \leq \varphi(t) \leq 1_{\{t \geq -2\eps\}}.
$$
By Proposition \ref{joint} applied to the function
$$
\Phi(t_1,\cdots,t_q) = \prod_{a=1}^q \varphi(t_a - \alpha^{|\CR_a|+1}),
$$
we conclude that
$$
\E_y \Phi(I_{\CR_1}(1_A;y),\cdots, I_{\CR_q}(1_A;y)) = \E_{\ve{y}}  \Phi(I_{\CL_1}(1_A;\ve{y}), \cdots, I_{\CL_q}(1_A;\ve{y})) + o(1).
$$
If $\ve{y}$ is a simultaneous popular difference for $\CL_1,\cdots,\CL_q$ in the sense that $I_{\CL_a}(1_A;\ve{y}) - \alpha^{|\CL_a|+1} \geq - \eps$ for every $a$, then $\Phi(I_{\CL_1}(1_A;\ve{y}),\cdots,I_{\CL_q}(1_A;\ve{y})) = 1$ by the construction of $\Phi$. Since there are at least $\beta p^m$ values of such $\ve{y}$, we have
$$
\E_{\ve{y}}  \Phi(I_{\CL_1}(1_A;\ve{y}), \cdots, I_{\CL_q}(1_A;\ve{y})) \geq \beta,
$$
and therefore
$$
\E_y \Phi(I_{\CR_1}(1_A;y),\cdots, I_{\CR_q}(1_A;y)) \geq \beta - o(1).
$$
On the other hand, if $\Phi(I_{\CR_1}(1_A;y),\cdots,I_{\CR_q}(1_A;y)) > 0$ for some $y \in \F_p$, then $I_{\CR_a}(1_A;y) - \alpha^{|\CR_a|+1} \geq -2\eps$ for every $a$ by the construction of $\Phi$. Since there are at least $(\beta - o(1))p$ values of such $y$, the desired conclusion follows.

	\section{Simultaneous popular difference of linear progressions}\label{sec:spd}
	
In this section we prove Propositions \ref{Main-prop}, \ref{Main2-prop}, and \ref{Main3-prop}. We allow all implied constants to depend on the given linear forms in the statements. It suffices to obtain the conclusion in Proposition \ref{Main-prop} with 
$$
I_{\CR}(1_A;\ve{y}) \geq \alpha^{|\CR|+1}-O(\eps),
$$
and similarly for Propositions \ref{Main2-prop} and \ref{Main3-prop}.

\subsection{The weight function via arithmetic regularity}
Let $A \subset \F_p$ be a subset with density $\alpha$.	
Let $\eta: \R_+\rightarrow\R_+$ be a growth function depending on $\eps$ which grows rapidly enough. By Proposition \ref{regularity}, we can find $\Gamma \subset \F_p$ with $|\Gamma| = O_\eps(1)$, $\rho \gg_{\eps}1$,
and a decomposition 
$$
1_A = f_{\str} + f_{\psr} + f_{\sml},
$$ 
where
		\begin{enumerate}
				\item[(a)] \(f_{\str}, f_{\psr}, f_{\sml}\) are all \(1\)-bounded,
				\item[(b)] \(f_{\str}\) is nonnegative, has mean \(\E f_{\str} =\alpha\), and obeys
				\[
				f_{\str}(x + d) = f_{\str}(x) + O(\eps)
				\]
				whenever \(x \in \F_p\) and \(d \in B := B(\Gamma, \rho)\),
				\item[(c)] \(\|\widehat{f_{\psr}}\|_{\infty} \le 1/\eta(|\Gamma|+\rho^{-1})\),
				\item[(d)] \(\|f_{\sml}\|_{2} \le\eps\).
			\end{enumerate}
Let $B_0 = B(\Gamma,\rho/2)$.
Define $w:\F_p\rightarrow \R_+$ and $\mu: \F_p^m\rightarrow \R_+$ by
$$
w = \Big(\frac{p}{|B_0|}\Big)^2 1_{B_0}*1_{B_0}, \qquad \mu(\ve{y}) = \prod_{r=1}^m w(y_r),
$$
where $\ve{y} = (y_1,\cdots,y_m)$. Thus $w$ and $\mu$ have average $1$, $w$ is supported on $B$, $\|w\|_{\infty} = O_{|\Gamma|,\rho}(1)$, and $\sum_{\xi \in \F_p} |\wh{w}(\xi)| = O_{|\Gamma|,\rho}(1)$.

\begin{prop}\label{prop:linear-0}
Let the notations and assumptions be as above. For any finite collection $\CL$ of nonzero linear forms in $m$ variables in $\Z_0[\ve{y}]$ with algebraic complexity $0$, we have 
$$
I_{\CL}(1_A;\ve{y}) \geq \alpha^{|\CL|+1} - O(\eps)
$$
for all $\ve{y} \in B^m$ outside an exceptional set of size $O(p^m/(\eps^2\eta(|\Gamma|+\rho^{-1})))$.
	\end{prop}

\begin{proof}
Using the decomposition from the regularity lemma, we can write $I_{\CL}(1_A;\ve{y})$ as a sum of terms of the form
$$
\E_{x \in \F_p} \prod_{L \in \CL \cup \{0\}} f_L(x + L(\ve{y})),
$$
where each $f_L$ is one of $f_{\str},f_{\psr},f_{\sml}$. First consider the term with $f_L = f_{\str}$ for every $L$. For $L \in \CL$ and $\ve{y} \in B^m$, by the previous observation \eqref{eq:translation-invariance} we have
$$
f_{\str}(x+L(\ve{y})) = f_{\str}(x) + O(\eps)
$$
for every $x \in \F_p$. It follows that
$$
\E_{x \in \F_p} \prod_{L \in \CL \cup \{0\}} f_{\str}(x + L(\ve{y})) \geq \E_{x \in \F_p} f_{\str}(x)^{|\CL|+1} - O(\eps)
\geq \alpha^{|\CL|+1} - O(\eps).
$$
Now consider the terms with $f_L = f_{\sml}$  for some $L$. In this case, by the triangle inequality we have
$$
\Big|\E_{x \in \F_p} \prod_{L \in \CL \cup \{0\}} f_L(x + L(\ve{y}))\Big| \leq \E_x |f_{\sml}(x)| \leq \eps.
$$
Finally, consider the terms with $f_L = f_{\psr}$ for some $L$. Then by \cite[Lemma 3.2]{CDH} we have
$$
\E_{\ve{y}} \Big|\E_{x \in \F_p} \prod_{L \in \CL \cup \{0\}} f_L(x + L(\ve{y}))\Big|^2 \leq \|\wh{f_{\psr}}\|_{\infty} \leq 1/\eta(|\Gamma|+\rho^{-1}).
$$
(This step uses the assumption that $\CL$ has complexity $0$, so that the pattern obtained after expanding the square on the left-hand side has complexity $1$.) Thus in this case we have
$$
 \Big|\E_{x \in \F_p} \prod_{L \in \CL \cup \{0\}} f_L(x + L(\ve{y}))\Big| \leq \eps
 $$
for all but at most $1/(\eps^2\eta(|\Gamma|+\rho^{-1}))$ proportion of $\ve{y}$. The desired conclusion follows.
\end{proof}
	
	\begin{prop}\label{prop:linear-1}
Let the notations and assumptions be as above. For any finite collection $\CL$ of nonzero linear forms in $m$ variables in $\Z_0[\ve{y}]$ with algebraic complexity at most $1$, we have 
$$
\E_{\ve{y}} \mu(\ve{y}) I_\CL(1_A;\ve{y}) \geq \alpha^{|\CL|+1} - O(\eps).
$$
In particular, there are $\gg_{|\Gamma|,\rho}\eps p^m$ values of $\ve{y} \in B^m$ such that
$$
I_{\CL}(1_A;\ve{y}) \geq \alpha^{|\CL|+1}-O(\eps).
$$
	\end{prop}

\begin{proof}
Using the decomposition from the regularity lemma, we can write the left-hand side as a sum of terms of the form
$$
\E_{x,\ve{y}} \mu(\ve{y}) \prod_{L \in \CL\cup\{0\}} f_L(x+L(\ve{y})),
$$
where each $f_{L}$ is one of $f_{\str}, f_{\psr}, f_{\sml}$. First consider the term with $f_L = f_{\str}$. As in the proof of Proposition \ref{prop:linear-0}, we have
$$
f_{\str}(x + L(\ve{y})) = f_{\str}(x) + O(\eps)
$$
for every $L \in \CL$, $x \in \F_p$, and $\ve{y}$ in the support of $\mu$. It follows that
$$
\E_{x,\ve{y}} \mu(\ve{y}) \prod_{L \in \CL\cup\{0\}} f_{\str}(x+L(\ve{y})) \geq \E_{x,\ve{y}} \mu(\ve{y})f_{\str}(x)^{|\CL|+1} - O(\eps) \geq \alpha^{|\CL|+1} - O(\eps).
$$
Now consider the terms with $f_L =f_{\sml}$ for some $L$. In this case, by the triangle inequality, we have
$$
\Big|\E_{x,\ve{y}} \mu(\ve{y}) \prod_{L \in \CL\cup\{0\}} f_L(x+L(\ve{y}))\Big| \leq \E_{x,\ve{y}} \mu(\ve{y}) |f_{\sml}(x+L(\ve{y}))|\leq \eps.
$$
Finally, consider the terms with $f_L = f_{\psr}$ for some $L$. By Fourier expansion, we can write
$$
\E_{x,\ve{y}} \mu(\ve{y}) \prod_{L \in \CL\cup\{0\}} f_L(x+L(\ve{y})) = \sum_{\ve{\xi}\in\F_p^m} \wh{\mu}(\ve{\xi}) \E_{x,\ve{y}} e_p(\xi\cdot \ve{y})\prod_{L \in \CL\cup\{0\}} f_L(x+L(\ve{y})).
$$
Since $\CL$ has algebraic complexity at most $1$, by a twisted form of the generalized von Neumann theorem (Proposition \ref{prop:twisted-gvn}) applied to the inner average over $x,\ve{y}$, we can bound the expression above by
$$
\Big(\sum_{\ve{\xi}\in\F_p^m} |\wh{\mu}(\ve{\xi})|\Big) \|\wh{f_{\psr}}\|_{\infty}^c  \ll_{|\Gamma|,\rho} \eta(|\Gamma|+\rho^{-1})^{-c}
$$
for some constant $c>0$. This can be made at most $\eps$ by choosing $\eta$ to grow fast enough. This shows that
$$
\E_{\ve{y}} \mu(\ve{y}) I_\CL(1_A;\ve{y}) \geq \alpha^{|\CL|+1} - O(\eps).
$$
Let $D$ be the set of $\ve{y} \in B^m$ such that $I_{\CL}(1_A;\ve{y}) \geq \alpha^{|\CL|+1}-C\eps$ for some large enough constant $C$. Then we have
$$
\alpha^{|\CL|+1}-O(\eps) \leq \E_{\ve{y}} \mu(\ve{y}) (\alpha^{|\CL|+1}-C\eps) + \E_{\ve{y}} \mu(\ve{y})1_{\ve{y}\in D}.
$$
It follows that
$$
\sum_{\ve{y} \in D} \mu(\ve{y}) \geq \eps p^m
$$
once $C$ is chosen large enough. Since $\|\mu\|_{\infty}\ll_{|\Gamma|,\rho}1$, we have $|D| \gg_{|\Gamma|,\rho} \eps p^m$ as desired.
\end{proof}

\subsection{Proof of Proposition \ref{Main-prop}}

Under the assumption that $\Span(\CL_1)\oplus\cdots\oplus\Span(\CL_\ell)$ is a direct sum, we may assume after a linear change of variables and clearing denominators that the linear forms in each $\CL_i$ only depend on the variables $y_r$ with $r \in \pi_i$, where $\pi_1,\cdots,\pi_{\ell}$ are pairwise disjoint subsets of $[m]$. By Proposition \ref{prop:linear-1} applied to each $\CL_i$, we can find $\gg_{|\Gamma|,\rho}\eps p^m$ values of $\ve{y}_i \in B^m$ such that
$$
I_{\CL_i}(1_A;\ve{y}_i) \geq \alpha^{|\CL_i|+1} - O(\eps).
$$
Since the linear forms in $\CL_i$ depend only on the variables $y_r$ with $r \in \pi_i$, the value $I_{\CL_i}(1_A;\ve{y}_i)$ depends only on the $r$-th component of $\ve{y}_i$ with $r \in \pi_i$. Since $\pi_1,\cdots,\pi_{\ell}$ are pairwise disjoint, we can find $\gg_{|\Gamma|,\rho} \eps^{\ell}p^m$ values of $\ve{y} \in B^m$ such that
$$
I_{\CL_i}(1_A;\ve{y}) \geq \alpha^{|\CL_i|+1} - O(\eps)
$$
for every $i$. Combining this with Proposition \ref{prop:linear-0} completes the proof, once $\eta$ is chosen to grow fast enough.

\subsection{Proof of Proposition \ref{Main2-prop}}

By Proposition \ref{prop:linear-1} we have
$$
I_{\CL}(1_A;y) \geq \alpha^{|\CL|+1} - O(\eps)
$$
for $\gg_{\eps}p$ values of $y \in \F_p$. Let $D$ be this set of popular differences $y$ for $\CL$. For a vector $\ve{y} = (y_1,\cdots,y_t) \in D^t$ for some $1 \leq t \leq m$, we say that $\ve{y}$ is $t$-\emph{good} if 
$$
I_{\CL_{[t]}}(1_A;\ve{y}) \geq \alpha^{t|\CL|+1} - C\eps,
$$
where $C$ is a sufficiently large constant.
Otherwise we say that $\ve{y}$ is $t$-\emph{bad}. Thus our goal is to find $\ve{y} = (y_1,\cdots,y_m) \in D^{m}$ such that $\ve{y}_I = (y_i)_{i \in I}$ is $|I|$-good for every nonempty $I \subset [m]$.

Let $R_{m} < R_{m-1} < \cdots < R_2 < R_1$ be a rapidly increasing sequence of large constants in terms of $\eps$. Pick an arbitrary subset $D_1 \subset D$ with $|D_1|=R_1$. First consider the complete graph whose vertex set is $D_1$, and each edge $\ve{y}=(y_1,y_2)$ is colored red or blue depending on whether $\ve{y}$ is $2$-good or not. Since $|D_1|=R_1$ is sufficiently large in terms of $R_2$, Ramsey theory implies that we can find a subset $D_2 \subset D_1$ with $|D_2| = R_2$, such that all edges in $D_2$ are monochromatic. Next consider the complete $3$-uniform hypergraph whose vertex set is $D_2$, and each $3$-edge $\ve{y}=(y_1,y_2,y_3)$ is colored red or blue depending on whether $\ve{y}$ is $3$-good or not. Since $|D_2|=R_2$ is sufficiently large in terms of $R_3$, Ramsey theory implies that we can find a subset $D_3 \subset D_2$ with $|D_3|=R_3$, such that all the $3$-edges in $D_3$ are monochromatic. Continue in this fashion to produce $D_4,\cdots,D_{m}$, with $|D_t|=R_t$ for all $2 \leq t \leq m$, such that all the $t$-edges in the $t$-uniform complete hypergraph with vertex set $D_t$ are monochromatic.

For each $2 \leq t \leq m$, let $\ve{y} = (y_1,\cdots,y_t)$ be a $t$-edge with distinct coordinates $y_1,\cdots,y_t \in D_t$. By the construction above, either all such $\ve{y}$ are $t$-good or all are $t$-bad. We prove that they must all be good. In fact, consider the average
$$
\E_{\ve{y} \in D_t^t}  I_{\CL_{[t]}}(1_A;\ve{y}) = \E_{x} 1_A(x) \Big(\E_{y \in D_t} \prod_{L \in \CL} 1_A(x+L(y))\Big)^t.
$$
Since $I_{\CL}(1_A;y) \geq \alpha^{|\CL|+1}-O(\eps)$ for $y \in D_t$, we have
$$
\begin{aligned}
\alpha^{|\CL|+1} - O(\eps) &\leq \E_x 1_A(x) \E_{y \in D_t}\prod_{L \in\CL} 1_A(x+L(y)) \\
&\leq \Big(\E_x 1_A(x)\Big)^{(t-1)/t} \Big[\E_x 1_A(x) \Big(\E_{y \in D_t} \prod_{L \in \CL} 1_A(x+L(y))\Big)^t \Big]^{1/t}
\end{aligned}
$$
by H\"{o}lder's inequality. It follows that 
$$
\E_{\ve{y} \in D_t^t}I_{\CL_{[t]}}(1_A;\ve{y}) \geq \alpha^{t|\CL|+1}-O(\eps).
$$
If all $\ve{y} \in D_t^t$ with distinct coordinates are $t$-bad, then their contribution to the above average is at most $\alpha^{t|\CL|+1}-C\eps$. On the other hand, the contribution from those $\ve{y} \in D_t^t$ with at least two equal coordinates is at most $t^2/|D_t| \leq \eps$ since $|D_t|$ is large. This leads to a contradiction once $C$ is large enough.

It follows that any $\ve{y} = (y_1,\cdots,y_m)$ with distinct $y_1,\cdots,y_m \in D_m$ is a popular difference for $\CL_I$ for any nonempty $I \subset [m]$. Indeed, if $I=\{i\}$ is a singleton then $y_i \in D$ and hence $I_{\CL}(1_A;y_i) \geq \alpha^{|\CL|+1}-O(\eps)$. If $t=|I| \geq 2$, then $\ve{y}_I \in D_m^t$ and so $\ve{y}_I \in D_t^t$ is $t$-good, and hence $I_{\CL_I}(1_A;\ve{y}) \geq \alpha^{t|\CL|+1}-C\eps$.

In this manner we produce one simultaneous popular difference $\ve{y} \in D_1^m$ for every $D_1 \subset D$ with $|D_1| = R_1$. By double counting $(\ve{y}, D_1)$, where $\ve{y}$ is a simultaneous popular difference, $D_1 \subset D$ with $|D_1|=R_1$, and $\ve{y} \in D_1^m$, we conclude that the total number of such $\ve{y}$ is at least
$$
\binom{|D|}{R_1} / |D|^{R_1-m} \gg_{\eps} |D|^m \gg_{\eps} p^m,
$$
as desired.

\subsection{Proof of Proposition \ref{Main3-prop}}

We start with the first pair $\CL_1=\{y,2y\}$ and $\CL_2 = \{y,2y,z\}$.
By Proposition \ref{prop:linear-1}, there are $\gg_{\eps}p$ values of $y \in \F_p$ such that $I_{\CL_1}(1_A;y) \geq \alpha^3-O(\eps)$. For such $y$ we have
$$
\E_{z \in \F_p} I_{\CL_2}(1_A;y,z) = \E_x 1_A(x) 1_A(x+y) 1_A(x+2y) \E_z 1_A(x+z) = \alpha I_{\CL_1}(1_A;y) \geq \alpha^4-O(\eps).
$$
Hence for each such $y$ there are at least $\eps$ proportion of $z \in \F_p$ such that $I_{\CL_2}(1_A;y,z) \geq \alpha^4-O(\eps)$. This completes the proof.

Now consider the second pair $\CL_1=\{z,2z\}$ and $\CL_2=\{y,z,y+z\}$. By Proposition \ref{Main-prop} applied to the two linear patterns $\CL_1 = \{z,2z\}$ and $\CR = \{z\}$, we obtain $\gg_{\eps}p$ values of $z\in\F_p$ satisfying
$$
I_{\CL_1}(1_A;z) \geq \alpha^3 - \eps \ \ \text{and}\ \  I_{\CR}(1_A;z) = \E_x 1_A(x) 1_A(x+z) \geq \alpha^2 - \eps.
$$
For such $z$ we have
$$
\begin{aligned}
\E_y I_{\CL_2}(1_A;y,z) &= \E_{x,y} 1_A(x) 1_A(x+y)1_A(x+z) 1_A(x+y+z) \\
&= \E_{x,u} 1_A(x) 1_A(x+z) 1_A(u) 1_A(u+z) = I_{\CR}(1_A;z)^2 \geq \alpha^4-O(\eps).
\end{aligned}
$$
where $u = x+y$. Hence for each such $z$ there are at least $\eps$ proportion of $y \in \F_p$ such that $I_{\CL_2}(1_A;y,z) \geq \alpha^4 - O(\eps)$. This completes the proof.

\section{Constructions for $(x, x+y, x+2y)$ and $(x,x+y^2)$}\label{sec:nspd1}

In this section we prove Theorem \ref{counterex}, which will be deduced from the following proposition.

\begin{prop}\label{nspd1}
Let $\CP_1 = \{y, 2y\}$ and $\CP_2 = \{y^2\}$. For sufficiently large prime $p$, there exists a function $f: \F_p\to [0,1]$ with $\E f = 1/2 + o(1)$, such that
$$
\min\Big\{I_{\CP_1}(f;y) - \frac{1}{8}, I_{\CP_2}(f;y) - \frac{1}{4}\Big\} \leq -c
$$
for some absolute constant $c>0$ and all but $O(1)$ values of $y \in \F_p$.
\end{prop}

In Section \ref{sec51} we give the construction in Proposition \ref{nspd1} using a trigonometric polynomial. An unfortunate side effect of the construction is that there are $O(1)$ exceptional values of $y$ which have to be ruled out. In Section \ref{sec52} we deduce Theorem \ref{counterex} from Proposition \ref{nspd1} via a modification of the standard probabilistic argument which handles these exceptional $y$.

\subsection{Proof of Proposition \ref{nspd1}}\label{sec51}
Let $\eps > 0$ be a sufficiently small constant.
Define $h,q,f: \F_p\rightarrow\R$ by $f(x) = h(x) + q(x)$, and
$$
h(x) = \frac{\eps^2}{2}(e_p(2x) + e_p(-2x)) = \eps^2 \cos\Big(\frac{4\pi x}{p}\Big), \ \ q(x) = \sum_{n \in \Z}c_n e_p(nx^2),
$$
where
$$
c_0 = \frac{1}{2}, \ \ c_{\pm 1} = c_{\pm 4} = \eps, \ \ c_{\pm 2} = -\frac{\eps}{2},
$$
and all other $c_n=0$. Thus $f$ takes values in $[0,1]$ once $\eps$ is small enough, and $\E f = 1/2+o(1)$ by the Gauss sum estimate.

\begin{lmm}\label{lem51}
Let the notations and assumptions be as above. Then we have
$$
I_{\CP_1}(f; y) =  \frac{1}{8} + \eps^3\Big[-\cos\Big(\frac{4\pi y^2}{p}\Big) + \frac{1}{2} \cos\Big(\frac{8\pi y^2}{p}\Big)\Big] + O(\eps^4 + p^{-1/2})
$$
for all but $O(1)$ values of $y \in \F_p$.
\end{lmm}

\begin{proof}
We can split $I_{\CP_1}(f;y)$ into eight terms of the form
$$
T(f_0, f_1, f_2) := \E_{x \in \F_p} f_0(x) f_1(x+y) f_2(x+2y),
$$
where $f_0,f_1,f_2 \in \{h,q\}$. If there are at least two copies of $h$ among $f_0,f_1,f_2$, then $|T(f_0,f_1,f_2)| \leq \eps^4$ since $\|h\|_{\infty} \leq \eps^2$. If exactly one of $f_0,f_1,f_2$ is $h$, say $f_0=h$, then
$$
T(h,q,q) = \frac{\eps^2}{2} \sum_{n_1,n_2 \in \Z} c_{n_1}c_{n_2}\E_{x \in \F_p}  e_p(n_1(x+y)^2 + n_2(x+2y)^2) (e_p(2x) + e_p(-2x)).
$$
If $n_1 +n_2 \neq 0$, then the inner average over $x$ is $O(p^{-1/2})$. Hence
$$
T(h,q,q) = \frac{\eps^2}{2} \sum_{n \in \Z} c_n^2 \E_{x \in \F_p} e_p(n(2xy+3y^2)) (e_p(2x) + e_p(-2x)) + O(p^{-1/2}).
$$
The inner average over $x$ vanishes unless $ny+1=0$ or $ny-1=0$. Hence $T(h,q,q) = O(p^{-1/2})$ for all but $O(1)$ values of $y$. A similar argument applies to $T(q,h,q)$ and $T(q,q,h)$.

It remains to compute the main term $T(q,q,q)$. We have
$$
T(q,q,q) = \sum_{n_0,n_1,n_2 \in \Z} c_{n_0}c_{n_1}c_{n_2} \E_{x \in \F_p} e_p(n_0x^2 + n_1(x+y)^2 + n_2(x+2y)^2).
$$
The inner average over $x$ is $O(p^{-1/2})$ unless $n_0+n_1+n_2=0$ and $(2n_1+4n_2)y = 0$. For $y \neq 0$, this happens only if $(n_0,n_1,n_2) = (n, -2n, n)$ for $n \in \{0,\pm 1,\pm 2\}$. Hence
$$
\begin{aligned}
T(q,q,q) &= \sum_{n \in \{0,\pm 1,\pm 2\}} c_n^2c_{-2n} e_p(2ny^2) + O(p^{-1/2}) \\
&= \frac{1}{8} + \eps^3\Big[-\cos\Big(\frac{4\pi y^2}{p}\Big) + \frac{1}{2} \cos\Big(\frac{8\pi y^2}{p}\Big)\Big] + O(p^{-1/2}).
\end{aligned}
$$
This completes the proof of the lemma.
\end{proof}

\begin{lmm}\label{lem52}
Let the notations and assumptions be as above. Then we have
$$
I_{\CP_2}(f; y) =  \frac{1}{4} + \frac{\eps^4}{2} \cos\Big(\frac{4\pi y^2}{p}\Big) + O(p^{-1/2})
$$
for all $y \in \F_p\setminus\{0\}$.
\end{lmm}

\begin{proof}
We can split $I_{\CP_2}(f;y)$ into four terms of the form
$$
T(f_0,f_1) := \E_{x \in \F_p} f_0(x) f_1(x+y^2),
$$
where $f_0,f_1 \in \{h,q\}$. For $f_0=f_1=q$, we have
$$
T(q,q) = \sum_{n_0,n_1 \in \Z} c_{n_0}c_{n_1} \E_{x \in \F_p} e_p(n_0x^2 +n_1(x+y^2)^2).
$$
The inner average is $O(p^{-1/2})$ unless $n_0+n_1=0$ and $2n_1y^2 = 0$. For $y \neq 0$, this happens only if $n_0=n_1=0$, and hence $T(q,q) = 1/4 + O(p^{-1/2})$.

For $f_0=h$ and $f_1=q$, we have
$$
T(h,q) = \frac{\eps^2}{2}\sum_{n \in \Z} c_n \E_{x \in \F_p} e_p(n(x+y^2)^2) (e_p(2x) + e_p(-2x)).
$$
The inner average over $x$ is $O(p^{-1/2})$ when $n \neq 0$, and vanishes when $n=0$. A similar argument applies to $T(q,h)$.

Finally, for $f_0=f_1=h$ we have
$$
\begin{aligned}
T(h,h) &= \frac{\eps^4}{4} \E_{x \in \F_p} (e_p(2x) + e_p(-2x)) (e_p(2x+2y^2) + e_p(-2x-2y^2)) \\
&= \frac{\eps^4}{4} (e_p(2y^2) + e_p(-2y^2)) = \frac{\eps^4}{2} \cos\Big(\frac{4\pi y^2}{p}\Big).
\end{aligned}
$$
This completes the proof of the lemma.
\end{proof}

The proof of Proposition \ref{nspd1} can now be completed as follows.
Let $u = \cos(4\pi y^2/p)$. If $u \leq -1/4$, then Lemma \ref{lem52} implies that
$$
I_{\CP_2}(f;y) \leq \frac{1}{4} - \frac{\eps^4}{8} + O(p^{-1/2})
$$
for $y \neq 0$. If $u \geq -1/4$, then
$$
-\cos\Big(\frac{4\pi y^2}{p}\Big) + \frac{1}{2} \cos\Big(\frac{8\pi y^2}{p}\Big) = -u + u^2 - \frac{1}{2} = \Big(u-\frac{1}{2}\Big)^2 - \frac{3}{4} \leq -\frac{3}{16}.
$$
Hence Lemma \ref{lem51} implies that
$$
I_{\CP_1}(f;y) \leq \frac{1}{8} -\frac{3}{16}\eps^3 + O(\eps^4 + p^{-1/2})
$$
for all but $O(1)$ values of $y$. The proof is completed by taking $\eps$ small enough.

\subsection{Deduction of Theorem \ref{counterex} from Proposition \ref{nspd1}}\label{sec52}

We summarize this step in the form of the following lemma, as it will be used again in Section \ref{sec:nspd3}. Clearly Theorem \ref{counterex} follows by combining Proposition \ref{nspd1} with Lemma \ref{lem:functions-to-sets}.

For a finite collection $\CP = \{P_1,\cdots,P_k\}$ of nonzero polynomials in $\Z_0[y]$, we say that $y \in \F_p$ is \emph{admissible} for $\CP$ if the $k+1$ values $0, P_1(y),\cdots,P_k(y)$ are all distinct. The expected value of $I_{\CP}(1_A; y)$ for a random subset $A\subset \F_p$ of density $\alpha$ is $\alpha^{|\CP|+1}$ only when $y$ is admissible for $\CP$. In the examples we consider in Theorems \ref{counterex}, \ref{counterex2}, and \ref{no-spd-linear}, all nonzero values of $y$ are admissible.

\begin{lmm}\label{lem:functions-to-sets}
Let $\CP_1,\CP_2$ be two finite collections of nonzero polynomials in $\Z_0[y]$. Suppose that there are constants $\alpha,\delta > 0$ such that for sufficiently large $p$ we have $f: \F_p \rightarrow [0,1]$ with $\E f = \alpha + o(1)$ satisfying
$$
\min\Big\{I_{\CP_1}(f;y) - \alpha^{|\CP_1|+1}, I_{\CP_2}(f;y) - \alpha^{|\CP_2|+1}\Big\} \leq -\delta
$$
for all but $O(1)$ values of $y \in \F_p$. Then there are constants $\alpha_0,c>0$ depending on $\CP_1,\CP_2,\alpha,\delta$ such that for sufficiently large $p$ we have $A \subset \F_p$ with density at least $\alpha_0$ satisfying
$$
\min\Big\{I_{\CP_1}(1_A;y) - \alpha_0^{|\CP_1|+1}, I_{\CP_2}(1_A;y) - \alpha_0^{|\CP_2|+1}\Big\} \leq -c
$$
for all $y \in \F_p$ which are admissible for both $\CP_1$ and $\CP_2$.
\end{lmm}

\begin{proof}
We allow all implied constants to depend on $\CP_1,\CP_2$.
Let $E$ be the exceptional set of values of $y$ from the hypothesis which are admissible.
Let $\tau > 0$ be a sufficiently small constant in terms of $\delta$. Define a random subset $B \subset \F_p$ by including each element $x \in \F_p$ in $B$ independently with probability $\tau f(x)$. For each $y \in E$, since $y$ is admissible for $\CP_1$, we can pick a polynomial $P_y \in \CP_1$ such that $P_y(y) \neq 0$. Let  $T = \{P_y(y): y \in E\}$. Define $A \subset B$ by
$$
A = \{x \in B: x+t \notin B\text{ for all }t \in T\}.
$$
The construction ensures that $I_{\CP_1}(1_A;y) = 0$ for every $y \in E$. Hence it suffices to prove the conclusion for admissible $y \notin E$.

To estimate the size $|A|$, note that $x \in A$ with probability
$$
\tau f(x) \prod_{t \in T} (1- \tau f(x+t)) \geq \tau f(x) (1-\tau)^{|T|} \geq \tau(1-O(\tau))f(x).
$$
Hence the expected size of $A$ is at least $\alpha \tau (1-O(\tau))p$. Note that changing the outcome of the random event $x \in B$ for one value of $x$ can change the size $|A|$ by at most $O(1)$. Thus it follows from the bounded differences inequality \cite{McD89} that
\begin{equation}\label{eq:Asize}
|A| \geq \alpha_0 p
\end{equation}
with probability $1-o(1)$, for some $\alpha_0 = \alpha\tau(1-O(\tau))$. On the other hand, by the bounded differences inequality again, we have for admissible $y$ and $i \in \{1,2\}$ that
\begin{equation}\label{eq:IC-expectation}
I_{\CP_i}(1_B;y) \leq \E I_{\CP_i}(1_B;y) + O(p^{-1/4}) \leq \tau^{|\CP_i|+1} I_{\CP_i}(f;y) + O(p^{-1/4})
\end{equation}
with probability $1- O(p^{-2})$. Thus with positive probability we have \eqref{eq:Asize} and \eqref{eq:IC-expectation} simultaneously for all admissible $y$.
 For admissible $y \notin E$, we have by hypothesis that
$$
I_{\CP_i}(f;y) \leq \alpha^{|\CP_i|+1} - \delta
$$
for some $i \in \{1,2\}$. This leads to
$$
I_{\CP_i}(1_A;y) \leq I_{\CP_i}(1_B;y) \leq \tau^{|\CP_i|+1} (\alpha^{|\CP_i|+1}-\delta) + O(p^{-1/4}) \leq \alpha_0^{|\CP_i|+1}-c
$$
for some constant $c>0$ by using \eqref{eq:Asize} and choosing $\tau$ small enough in terms of $\delta$.
\end{proof}

\section{Constructions for $(x,x+y,x+2y)$ and $(x,x+y^2,x+2y^{2})$}\label{sec:nspd2}

In this section we prove Theorem \ref{counterex2}. Let $\alpha > 0$ be small.
Let $M = M(\alpha)$ be a sufficiently large constant such that we can find a subset $J \subset [M]$ with $|J| \geq 100\alpha^{1/2}M$ with no nontrivial $3$-APs. The construction of Behrend \cite{B46} ensures that we can take $M = \lfloor\exp(c (\log \alpha^{-1})^2)\rfloor$ for some absolute constant $c>0$.

Let $p$ be sufficiently large in terms of $M$, and set $K:=\lfloor p/(4M)\rfloor$. Consider the following disjoint intervals of length $\lfloor K/10\rfloor$:
$$
 I_{j}:=\Big\{(j-1)K+1,\dots,(j-1)K + \lfloor K/10\rfloor\Big\} \quad(\operatorname{mod}\, p)
 $$ 
for $1\le j\le M$. Define
$$
B = \bigcup_{j \in J} I_j,\qquad
A = \Big\{x \in \F_p: x \in B, x^2 \in B\}.
$$
Note that $|B| \sim |J|K/10 \geq 10\alpha^{1/2}MK\geq 2\alpha^{1/2}p$. The following lemma shows that the density of $A$ is roughly the square of the density of $B$, and hence at least $\alpha$, using a standard Fourier argument.
 
\begin{lmm}\label{a}
The set $A\subset \F_p$ constructed above satisfies $|A| \geq \alpha p$.
\end{lmm}
\begin{proof}
By the Fourier expansion of $1_{B}$, we have 
$$
		\frac{|A|}{p}=\E_{x}1_{B}(x)1_{B}(x^{2})=\sum_{r,s\in\F_{p}}\widehat{1_{B}}(r)\widehat{1_{B}}(s)\E_{x}e_{p}(rx+sx^{2})
$$
The term from $r=s=0$ contributes $\wh{1_B}(0)^2 \geq (|B|/p)^2 \geq 4\alpha$. Using the Gauss sum estimate, we can bound all the remaining terms by
$$
\ll p^{-1/2} \sum_{r,s\in\F_p} |\wh{1_B}(r)\wh{1_B}(s)| = p^{-1/2} \|\wh{1_B}\|_{L^1}^2.
$$
Since $B$ is a union of at most $M$ intervals and for each interval $I \subset \F_p$ we have $\|\wh{1_I}\|_{L^1} \ll \log p$, it follows that $\|\wh{1_B}\|_{L^1} \ll M \log p$, and thus the expression above is
$$
\ll p^{-1/2} M^2 (\log p)^2 \leq \alpha
$$
since $p$ is sufficiently large. This completes the proof.
\end{proof}

\begin{lmm}\label{approx 3ap}
If three elements $x,y,z \in B$ satisfy 
$$
\Big\|\frac{x+z-2y}{p}\Big\|_{\mathbb{T}}<\frac{1}{10M},
$$ 
then $x,y,z$ must lie in the same interval $I_j$ for some $j \in J$.
\end{lmm}
\begin{proof}
	We can write
$$
x = (a-1)K + u, \qquad y = (b-1)K+v, \qquad z=(c-1)K+w,
$$
where $a,b,c\in J$ and $1 \leq u,v,w \leq K/10$. Then
$$
\Big\|\frac{(a+c-2b)K}{p}\Big\|_{\T} = \Big\|\frac{u+w-2v}{p} - \frac{x+z-2y}{p}\Big\|_{\T} \leq \frac{4K}{10p} + \frac{1}{10M} \leq \frac{1}{5M}.
$$
Since $1 \leq a,b,c \leq M$, we have
$$
\Big|\frac{(a+c-2b)K}{p}\Big| \leq \frac{2MK}{p} \leq \frac{1}{2}.
$$
It follows that we must have
$$
\Big|\frac{(a+c-2b)K}{p}\Big| \leq \frac{1}{5M},
$$
and hence $|a+c-2b| < 1$, which forces $a+c-2b=0$. Since $J$ does not contain nontrivial $3$-APs, we have $a=b=c$, as desired.
\end{proof}

Let $T(y)$ be the normalized number of $3$-APs in $A$ with common difference $y$:
$$
T(y) = \E_{x \in \F_p} 1_A(x) 1_A(x+y) 1_A(x+2y).
$$
For $y \neq 0$, we will show that either $T(y^2) = 0$ or $T(y)$ is very small.
Suppose that $T(y^2) > 0$ for some $y \neq 0$. Then there exists $x_0 \in \F_p$ such that $x_0, x_0+y^2, x_0+2y^2 \in A$. By Lemma \ref{approx 3ap} applied to the three elements $x_0, x_0+y^2, x_0+2y^2$ (which all lie in $B$), we have $x_0, x_0+y^2, x_0+2y^2$ must lie in a common interval of length $\lfloor K/10\rfloor$. Hence
$$
\Big\|\frac{x_0^2+(x_0+2y)^2-2(x_0+y)^2}{p}\Big\|_\T \leq 2\Big\|\frac{y^2}{p}\Big\|_\T \leq \frac{K}{5p} < \frac{1}{10M}.
$$
Now if $x, x+y, x+2y \in A$ for some $x \in \F_p$, then $x^2, (x+y)^2, (x+2y)^2$ all lie in $B$, and hence we may apply Lemma \ref{approx 3ap} again to conclude that $x^2, (x+y)^2, (x+2y)^2$ must lie in a common interval of length $\lfloor K/10\rfloor$. This implies that
$$
\Big\| \frac{(x+y)^2-x^2}{p}\Big\|_\T = \Big\|\frac{2xy+y^2}{p}\Big\|_\T \leq \frac{K}{10p} < \frac{1}{20M}.
$$
Since $y \neq 0$, the number of $x \in \F_p$ satisfying the above inequality is at most $p/(10M)$, and hence $T(y) \leq 1/(10M)$. This completes the proof by our choice of $M$.

\section{Constructions for pairs of linear patterns}\label{sec:nspd3}

In this section we prove Theorem \ref{no-spd-linear}, which is a consequence of the following proposition by applying Lemma \ref{lem:functions-to-sets}.

\begin{prop}\label{nspd3}
Let $C_1, C_2 \subset \Z$ with each $C_i$ either a triple or an additive quadruple. Suppose that $C_1,C_2$ are not congruent.
For sufficiently large prime $p$, there exists a function $f: \F_p\to [0,1]$ with $\E f = 1/2 + o(1)$, such that
$$
\min\Big\{I_{C_1}(f;y) - 2^{-|C_1|}, I_{C_2}(f;y) - 2^{-|C_2|}\Big\} \leq -c
$$
for some constant $c>0$ depending on $C_1,C_2$ and all but $O_{C_1,C_2}(1)$ values of $y \in \F_p$.
\end{prop}

In this section we allow all implied constants to depend on $C_1,C_2$. 
We define $f: \F_p \rightarrow [0,1]$ by
$$
f(x) = \frac{1}{2} + G\Big(\frac{x}{p}, \frac{x^2}{p}\Big),
$$
where $G: \T^2\rightarrow [-1/2,1/2]$ is a real trigonometric polynomial (depending on $C_1,C_2$), whose Fourier transform $\wh{G}(m,n)$ is supported on $O(1)$ pairs $(m,n)$ and satisfies $\wh{G}(m,0)=0$ for all $m$ and $|\wh{G}(m,n)| \leq \eps$ for all $m,n$, where $\eps > 0$ is a sufficiently small constant. Hence $\|G\|_{\infty} = O(\eps)$ and $\E f= 1/2+o(1)$ by the Gauss sum estimate.

\subsection{$I_C(f;y)$ for triples and additive quadruples}\label{sec:IC}

We begin with writing $I_C(f;y)$ in terms of the Fourier coefficients $\wh{G}(m,n)$.

\begin{lmm}\label{lem:IC-triple}
	Let the notations and assumptions be as above.
	If $C = \{0, du, d(u+v)\}$, where $d,u,v$ are positive integers and $(u,v)=1$, then
	$$
	\begin{aligned}
		I_C(f;y) &= \frac{1}{8} + \sum_{n\in\Z\setminus\{0\}}\sum_{\substack{m_0,m_1,m_2 \in \Z \\m_0+m_1+m_2=0}} \wh{G}(m_0,vn) \wh{G}(m_1,-(u+v)n) \wh{G}(m_2,un)  \\
		&\qquad\qquad\qquad \cdot e_p( d[um_1+(u+v)m_2]y + d^2uv(u+v)n y^2) + O(p^{-1/2})
	\end{aligned}
	$$
	for all but $O(1)$ exceptional values of $y \in \F_p$.
\end{lmm}

\begin{proof}
Using the Fourier expansion $f(x) = \sum_{m,n\in\Z} c_{m,n} e_p(mx +nx^2)$, where 
$c_{0,0} = 1/2$ and $c_{m,n} = \wh{G}(m,n)$ for $(m,n) \neq (0,0)$, 
we obtain
$$
I_C(f;y) = \sum_{n_0,n_1,n_2} \sum_{m_0,m_1,m_2} c_{m_0,n_0} c_{m_1,n_1} c_{m_2,n_2}\cdot\E_{x \in \F_p}e_p(P_{m_0,m_1,m_2,n_0,n_1,n_2}(x)),
$$
where the quadratic polynomial $P_{m_0,m_1,m_2,n_0,n_1,n_2}(x)$ is given by
$$
m_0x + m_1(x+duy) + m_2(x+d(u+v)y) + n_0x^2 + n_1(x+duy)^2 + n_2(x+d(u+v)y)^2.
 $$
Examining its quadratic and linear coefficients, we see that the inner average over $x$ is $O(p^{-1/2})$ unless
$$
n_0+n_1+n_2 = m_0+m_1+m_2 + 2d[un_1+(u+v)n_2]y = 0.
$$
Since we may restrict to the $O(1)$ pairs $(m_0,n_0),(m_1,n_1),(m_2,n_2)$ with the nonzero Fourier coefficients, for all but $O(1)$ exceptional values of $y \in \F_p$ we may deduce  from the above identities that
$$
n_0+n_1+n_2 = m_0+m_1+m_2 = un_1 + (u+v)n_2 = 0.
$$
Since $(u,v)=1$, it follows that $(n_0,n_1,n_2) = (vn, -(u+v)n, un)$ for some $n \in \Z$. If $n_0=n_1=n_2=0$, then by the support of $\wh{G}$ we must also have $m_0=m_1=m_2=0$, which contributes the main term $1/8$. The desired conclusion follows after computing the constant polynomial $P_{m_0,m_1,m_2,n_0,n_1,n_2}(x)$ when $(n_0,n_1,n_2) = (vn, -(u+v)n, un)$ and $m_0+m_1+m_2=0$.
\end{proof}

\begin{lmm}\label{lem:IC-quadruple}
Let the notations and assumptions be as above.	
	If $C$ is an additive quadruple, then
	$$
	I_{C}(f;y) = \frac{1}{16} + \frac{1}{2}\sum_{\substack{B \subset C \\ |B|=3}} \Big(I_{B}(f;y)-\frac{1}{8}\Big) + O(\eps^4 + p^{-1/2})
	$$
	for all but $O(1)$ exceptional values of $y \in \F_p$.
\end{lmm}

\begin{proof}
From $f(x) = 1/2 + g(x)$ where $g(x) = G(x/p, x^2/p)$, we can expand
$$
I_C(f;y) = \sum_{B \subset C} 2^{|B|-|C|}\E_{x \in \F_p} \prod_{c \in B} g(x+cy).
$$
The term $B = \emptyset$ contributes $1/16$. The term $B = C$ is $O(\eps^4)$ since $\|g\|_{\infty} = O(\eps)$. The terms with $|B|=1$ contribute $O(p^{-1/2})$ by the Gauss sum estimate. The terms with $|B| = 3$ lead to
$$
\frac{1}{2} I_B(g;y) = \frac{1}{2} \Big(I_B(f;y) - \frac{1}{8}\Big) + O(p^{-1/2}),
$$
where the comparison between $I_B(f;y)$ and $I_B(g;y)$ follows from applying Lemma \ref{lem:IC-triple} to $I_B(f;y)$ and also applying the same argument to $I_B(g;y)$, which leads to the same formula without the term $1/8$. Finally, if $|B|=2$ with $B = \{b_1,b_2\}$ then we can show that
$$
\E_{x \in \F_p} g(x+b_1y) g(x+b_2y) = O(p^{-1/2})
$$
for all but $O(1)$ values of $y\in \F_p$.
This can be seen by either running the same argument as in the proof of Lemma \ref{lem:IC-triple}, or by enlarging $B = \{b_1,b_2\}$ to a triple $\{b_1,b_2,b_3\}$ and viewing the above as
$$
\E_{x \in \F_p} g(x+b_1y) g(x+b_2y) \ve{1}(x+b_3y)
$$
(where $\ve{1}$ denotes the constant function $1$), and finally using an asymmetric version of Lemma \ref{lem:IC-triple}. This completes the proof.
\end{proof}

\subsection{Assigning the Fourier coefficients of $G$}\label{sec:assign-fourier}

After translation, every triple can be written as $\{0, du, d(u+v)\}$, where $d,u,v$ are positive integers with $(u,v)=1$. We use the unordered pair $\{u,v\}$ (allowing $u=v=1$) to denote its \emph{shape}. Note that congruent triples have the same shape. After translation, every additive quadruple can be written as $\{0, du, dv, d(u+v)\}$, where $d,u,v$ are positive integers with $u<v$ and $(u,v)=1$. It has four subsets which are triples:
$$
\{0, du, dv\},  \qquad \{du, dv, d(u+v)\}, \qquad \{0, du, d(u+v)\}, \qquad \{0, dv, d(u+v)\}.
$$
The first two triples above are congruent and have the same shape $\{u,v-u\}$, and the last two triples above are also congruent and have the same shape $\{u,v\}$. We say that the additive quadruple has shapes $\{u,v-u\}$ and $\{u,v\}$. 

\begin{lmm}\label{lem:assign-Ghat}
	For every $\sigma = \{u,v\}$ which appears as the shape or one of the two shapes of $C_1$ and $C_2$, let $h_{\sigma}:\T\rightarrow\R$ be a real trigonometric polynomial of degree $O(1)$ with $\int h_{\sigma}=0$ and $\|h_{\sigma}\|_{\infty} = O(1)$. Let $L$ be a positive integer divisible by $uv(u+v)$ for all the shapes $\sigma = \{u,v\}$ involved. There exists a real trigonometric polynomial  $G: \T^2\rightarrow \R$ of degree $O(1)$, satisfying   $\wh{G}(m,0)=0$ for all $m$ and $|\wh{G}(m,n)| \leq \eps$ for all $m,n$, such that  for each shape $\sigma = \{u,v\}$ involved we have the identity
	$$
	\begin{aligned}
		&\sum_{n\in\Z\setminus\{0\}}\sum_{\substack{m_0,m_1,m_2 \in \Z \\m_0+m_1+m_2=0}} \wh{G}(m_0,vn) \wh{G}(m_1,-(u+v)n) \wh{G}(m_2,un) \\
		&\qquad \cdot e([um_1+(u+v)m_2]X + uv(u+v)n Y) = \eps^{7/2} h_{\sigma}(LY) + O(\eps^4)
	\end{aligned}
	$$
	for all $X, Y \in \T$.
\end{lmm}

\begin{proof}
	List all the $O(1)$ nonzero Fourier coefficients $\wh{h_\sigma}(k)$ for all the shapes $\sigma$ involved and $k>0$. Let $K$ be a sufficiently large positive integer. For each index $\lambda$ in this list, let $\wh{h_\sigma}(k)$ be the $\lambda$-th item in the list with $\sigma = \sigma_{\lambda}=\{u_{\lambda},v_{\lambda}\}$ and $k = k_{\lambda}> 0$, and let $n_\lambda = k_\lambda L/[u_\lambda v_\lambda(u_\lambda+v_\lambda)]$.  If $\sigma_{\lambda} \neq \{1,2\}$, then we set
	$$
	\wh{G}(K^\lambda n, n) = \begin{cases} \eps^{3/2} \overline{\wh{h_\sigma}(k)}, &\text{if }n = (u_\lambda+v_\lambda)n_\lambda,\\ \eps, &\text{if }n  \in \{u_\lambda n_\lambda, v_\lambda n_\lambda\}. \end{cases}
	$$
	If $\sigma_{\lambda} = \{1,2\}$, then we set
	$$
	\wh{G}(K^\lambda n, n) = \begin{cases} \eps^{3/2} \wh{h_\sigma}(k), &\text{if }n = n_\lambda,\\ \eps, &\text{if }n  \in \{2n_\lambda, 3n_\lambda\}. \end{cases}
	$$
	Set $\wh{G}(m,n) = \overline{\wh{G}(-m,-n)}$ and set all other Fourier coefficients to be $0$. With this choice, we have for all $\lambda$ that
	$$
	\wh{G}(K^\lambda v_\lambda n_\lambda, v_\lambda n_\lambda) \wh{G}(-K^\lambda (u_\lambda+v_\lambda)n_\lambda,-(u_\lambda+v_\lambda)n_\lambda) \wh{G}(K^\lambda u_\lambda n_\lambda, u_\lambda n_\lambda) = \eps^{7/2} \wh{h_{\sigma}}(k).
	$$
	Now we check that this function $G$ satisfies the desired conclusion. Let $\sigma = \{u,v\}$ be a shape involved. The summand vanishes unless
	$$
	m_0 = K^{\lambda_0} vn, \qquad m_1 = -K^{\lambda_1}(u+v)n, \qquad m_2 = K^{\lambda_2}un
	$$
	for some positive integers $\lambda_0,\lambda_1,\lambda_2$. Since $n \neq 0$ and $m_0+m_1+m_2=0$, this can only happen if $\lambda_0=\lambda_1=\lambda_2 = \lambda$, and hence the coefficient in front of $X$ is $um_1 + (u+v)m_2 = 0$. Thus the left-hand side becomes
	$$
2\operatorname{Re}	\sum_\lambda \sum_{n>0} \wh{G}(K^\lambda vn, vn) \wh{G}(-K^\lambda (u+v)n, -(u+v)n) \wh{G}(K^\lambda un, un) e(uv(u+v)nY),
	$$
after combining the terms from $\pm n$ together.
	By our choice of $\wh{G}(K^\lambda n, n)$, the summand vanishes unless
	$$
	un, vn, (u+v)n \in \Big\{ \frac{k_\lambda L}{u_\lambda v_\lambda}, \frac{k_\lambda L}{u_{\lambda}(u_\lambda+v_\lambda)}, \frac{k_\lambda L}{v_\lambda(u_\lambda+v_\lambda)}\Big\} = \{u_\lambda n_\lambda, v_\lambda n_\lambda, (u_\lambda+v_\lambda)n_\lambda\},
	$$
	where $n_{\lambda} = k_\lambda L/[u_\lambda v_\lambda(u_\lambda+v_\lambda)]$. This happens only when $\sigma = \sigma_{\lambda}$, or when $\sigma = \{1,1\}$ and $\sigma_{\lambda}= \{1,2\}$. In the case $\sigma=\sigma_{\lambda}$, we have $n = n_\lambda = k_{\lambda}L/[uv(u+v)]$, and their contribution to the left-hand side is
	$$
	2 \operatorname{Re}\sum_{\lambda:\sigma_\lambda=\sigma} \eps^{7/2} \wh{h_\sigma}(k_{\lambda}) e(k_\lambda LY) = \eps^{7/2}h_\sigma(LY).
	$$
	If $\sigma =\{1,1\}$ and $\sigma_\lambda = \{1,2\}$, then $n = n_{\lambda}$ and each such term contributes at most
	$$
	|\wh{G}(K^\lambda n_{\lambda},n_\lambda)^2 \wh{G}(-2K^\lambda n_{\lambda}, -2n_{\lambda})| \ll \eps^4.
	$$ 
	This completes the proof.
\end{proof}

By translation, we may assume that $0 \in C_1$ and $0 \in C_2$, and denote by $d_1$ and $d_2$ to be the greatest common divisor of all coefficients in $C_1$ and $C_2$, respectively. Let $L$ be a sufficiently large positive integer. By combining Lemmas \ref{lem:IC-triple}, \ref{lem:IC-quadruple}, and \ref{lem:assign-Ghat}, it follows that for $i \in \{1,2\}$ and all but $O(1)$ values of $y \in \F_p$ we have
$$
I_{C_i}(f;y) = 2^{-|C_i|} + \eps^{7/2}H_i\Big(\frac{Ld_i^2y^2}{p}\Big) + O(\eps^4+p^{-1/2}),
$$
where
$$
H_i = \begin{cases} h_{\sigma}, & \text{if }C_i\text{ is a triple with shape }\sigma, \\
	h_{\sigma}+h_{\tau}, & \text{if }C_i\text{ is an additive quadruple with shapes }\sigma,\tau. \end{cases}
$$
In the next subsection, we will construct $h_{\sigma}$ for those shapes $\sigma$ involved in $C_1,C_2$ to produce $H_1,H_2$ satisfying
\begin{equation}\label{eq:H1H2}
\min\{ H_1(d_1^2Y), H_2(d_2^2Y)\} \leq -1/2
\end{equation}
for all $Y \in \T$. This would complete the proof of Proposition \ref{nspd3}.

\subsection{Choosing the trigonometric polynomials $h_{\sigma}$}\label{sec:hsigma}

\begin{lmm}\label{min}
	For any positive integers $a,b$, there are real trigonometric polynomials $h_1,h_2:\T\rightarrow \R$ of degree $O_{a,b}(1)$ with $\int h_i =0$ and $\|h_i\|_{\infty} = O(1)$ for each $i \in \{1,2\}$, such that
	$$
	\min(h_1(ax), h_2(bx)) \le -\frac{1}{2}
	$$
	for all $x \in \T$. Moreover, if $a \neq b$ then we may additionally require $h_1=h_2$.
\end{lmm}

\begin{proof}
	For some $N$ large enough in terms of $a,b$, we use the Fej\'{e}r kernel $K_N:\T\rightarrow \R$,
	$$
	K_N(x) = \frac{1}{N} \Big|\sum_{j=0}^{N-1} e(jx)\Big|^2 = \frac{1}{N}\Big(\frac{\sin(\pi Nx)}{\sin (\pi x)}\Big)^2.
	$$
	Then $\int K_N = 1$, and we have the bound
	$$
	|K_N(x)| \ll \frac{1}{N\|x\|_{\T}^2}.
	$$
	Choose $\theta_1,\theta_2 \in \T$ such that $\|b\theta_1-a\theta_2\|_\T=1/2$. For example, take $(\theta_1,\theta_2) = (0, 1/(2a))$. Moreover, if $a\neq b$ then we may take $\theta_1 = \theta_2 = 1/(2|a-b|)$. Define
	$$
	h_i(x)  = K_N(x-\theta_i)-1
	$$
	for $i \in \{1,2\}$. Then $\int h_i = 0$, and $h_1=h_2$ if $a \neq b$. For any $x \in \T$, either one of $\|ax-\theta_1\|_\T$ and $\|bx-\theta_2\|_\T$ is at least $1/[2(a+b)]$, since otherwise we would have
	$$
	\|b\theta_1-a\theta_2\|_{\T} = \|-b(ax-\theta_1)+a(bx-\theta_2)\|_{\T} \le b\|ax-\theta_1\|_\T + a\|bx-\theta_2\|_\T < \frac{1}{2},
	$$
	contradicting our choice of $\theta_1,\theta_2$. It follows that
	$$
	\min(h_1(ax), h_2(bx)) =  O\Big(\frac{1}{N \max(\|ax-\theta_1\|_\T, \|bx-\theta_2\|_\T)^2}\Big) - 1 = O_{a,b}\Big(\frac{1}{N}\Big)-1,
	$$
	which can be made at most $-1/2$ by choosing $N$ large enough.
\end{proof}

Without loss of generality assume that $|C_1| \leq |C_2|$.
Apply Lemma \ref{min} with $(a,b) = (d_1^2, d_2^2)$ to produce $h_1,h_2$. For each shape $\sigma$ appearing as the shape or one of the two shapes of $C_1$ and $C_2$, we define $h_{\sigma}$ as follows:

\begin{itemize}
	\item If $C_1,C_2$ are triples with different shapes $\sigma$ and $\tau$, then set $h_{\sigma}=h_1$ and $h_{\tau}=h_2$.
	\item If $C_1,C_2$ are triples with the same shape $\sigma$, then we necessarily have $d_1 \neq d_2$ since they are not congruent, so $h_1=h_2$ from Lemma \ref{min}. Set $h_{\sigma} = h_1=h_2$.
	\item If $C_1$ is a triple with shape $\nu$ and $C_2$ is an additive quadruple with shapes $\sigma,\tau$, where $\nu$ is different from both $\sigma,\tau$, then set $h_{\nu} = h_1$, $h_{\sigma}=h_2$, $h_{\tau}=0$.
	\item If $C_1$ is a triple with shape $\sigma$ and $C_2$ is an additive quadruple with shapes $\sigma,\tau$, then set $h_{\sigma}=h_1$ and $h_{\tau} = h_2-h_1$.
	\item If $C_1,C_2$ are additive quadruples with different pairs of shapes, then we may choose $\sigma$ to be one of the two shapes of $C_1$ which is not any of the two shapes of $C_2$, and choose $\tau$ to be one of the two shapes of $C_2$ which is not any of the two shapes of $C_1$. Set $h_{\sigma}=h_1$, $h_{\tau}=h_2$, and set $h_\nu=0$ for the other shape(s) $\nu$ involved.
	\item If $C_1,C_2$ are additive quadruples with the same pair of shapes $\sigma,\tau$, then we necessarily have $d_1\neq d_2$ since they are not congruent, so $h_1=h_2$ from Lemma \ref{min}. Set $h_{\sigma}=h_1=h_2$ and $h_{\tau}=0$.
\end{itemize}

In all of the above cases, one can verify \eqref{eq:H1H2}, thus completing the proof of Proposition \ref{nspd3}.

 	\bibliographystyle{plain}
	\bibliography{spd}

\end{document}